\documentclass[11pt, a4]{amsart}

\usepackage[english]{babel}
\usepackage{amsmath,enumerate, amsfonts, amssymb,amsthm, xypic}
\usepackage[dvips]{graphics}
\usepackage{graphicx} 
\usepackage[dvipsnames]{xcolor}
\usepackage{hyperref}
\hypersetup{
	colorlinks=true,
	linkcolor={Blue},
	citecolor={red!60!black},
	urlcolor={cyan},
}
\usepackage{array}
\usepackage{float}
\usepackage{mathtools}
\usepackage{amsmath}
\usepackage{stmaryrd}

\newtheorem{thm}{Theorem}[section]
\newtheorem{lem}[thm]{Lemma}
\newtheorem{prop}[thm]{Proposition}
\newtheorem{cor}[thm]{Corollary}

\theoremstyle{definition}
\newtheorem{defi}[thm]{Definition}
\newtheorem{notation}[thm]{Notation}
\theoremstyle{remark} \newtheorem{ex}[thm]{Example}

\theoremstyle{remark} \newtheorem{rem}[thm]{Remark}
\newtheorem{conj}[thm]{Conjecture}

\title{Poles of real motivic zeta functions for curves}

\author{Théo Jaudon}

\address{Université de Rennes, IRMAR - UMR 6625, F-35000 Rennes, France}

\email{theo.jaudon@univ-rennes.fr}

\begin{document}

\begin{abstract}
    To a given real polynomial function $f \in \mathbb{R}[x_1,\dots,x_d]$, we associate real topological zeta functions $Z_{top,0}(f;s)$ and $Z_{top,0}^{\pm}(f;s) \in \mathbb{Q}(s)$, analogous to the topological zeta function of Denef and Loeser in the complex case. These functions are specializations of the real motivic zeta functions studied in \cite{goulw1} and \cite{campesato}. Therefore, these functions and their sets of poles are invariants of the blow-Nash equivalence. Using the approach of \cite{veys1}, we study the poles of these real topological zeta functions, as well as real motivic zeta functions, when $f$ is a real polynomial in two variables.
\end{abstract}

\maketitle

\section*{Introduction}
Let $f : (\mathbb{C}^d,0) \to (\mathbb{C},0)$ be a complex analytic function germ. In \cite{loes2}, Denef and Loeser associate to $f$ a rational function $Z_{top,0}(f;s) \in \mathbb{Q}(s)$ called the local topological zeta function of $f$. If $ \sigma : (X,\sigma^{-1}(0)) \to (\mathbb{C}^d,0)$ is an analytic modification such that the divisors $\text{div} (f \circ \sigma)$ and $\sigma^*(dx_1 \wedge \dots \wedge dx_d)$ are simultaneously normal crossings, the local topological zeta function is defined by 
$$Z_{top,0}(f;s) = \sum\limits_{I \subset J} \chi(E^{0}_I \cap \sigma^{-1}(0)) \underset{i \in I}{\prod}\frac{1}{\nu_i+sN_i}$$ 
where the integers $(\nu_i,N_i)_{i \in I}$ are the numerical data of the resolution and $(E_I^0)_{I \subset J}$ denotes the canonical stratification of $(f\circ\sigma)^{-1}(0) = \underset{j \in J}{\bigcup} E_j$ into smooth subvarieties.\\

The authors show that the above expression does not depend on the chosen resolution by interpreting $Z_{top,0}(f;s)$ as a certain limit of $p$-adic Igusa function. Nowadays, one can show that $Z_{top,0}(f;s)$ does not depend on the chosen resolution by using the weak factorization theorem \cite{wlod} or by viewing $Z_{top,0}(f;s)$ as a specialization of the motivic zeta function $Z_{mot}(f;\mathbb{L}^{-s})$, which is defined intrinsically (see, for example, \cite{loes1}).\\
Despite what the name suggests, $Z_{top,0}(f;s)$ is an analytic invariant of $f$ in a neighborhood of the origin, but it is generally not a topological invariant (see \cite{bartolo} for a counterexample).\\
By definition, the set of poles of $Z_{top,0}(f;s)$ is included in the set 
$$\{~-\frac{\nu_i}{N_i} \mid i \in J~\} \subset \mathbb{Q}_{< 0}.$$

However, this list of candidate poles generally contains a significant number of false poles, and determining which of them are the true poles of $Z_{top,0}(f;s)$ is a very difficult problem. In this direction, the monodromy conjecture \cite{veys3} predicts that the poles of $Z_{top,0}(f;s)$ must satisfy a topological condition, thereby allowing one to restrict the set of true poles of $Z_{top,0}(f;s)$.

\begin{thm}[\cite{milnor1}]
    For $a \in \{f=0\}$, with $0 < \delta \ll \varepsilon \ll 1$, let $D_{\delta}^* \subset \mathbb{C}$ denote the punctured open disc of radius $\delta$ and let $B_{a,\varepsilon} \subset \mathbb{C}^d$ be the closed ball centered at $a$ with radius $\varepsilon$. The restriction 

    $$f : f^{-1}(D_{\delta}^*) \cap B_{a,\varepsilon} \to D_{\delta}^* $$
    is a locally trivial smooth fibration. Its fiber is denoted by $\mathcal{F}_{f,a}$ and is called the Milnor fiber of $f$ at $a$. It is a compact, orientable smooth manifold with boundary of dimension $2(d-1)$.\\
    Moreover, a generator of $\pi_1(D_{\delta}^* )$ induces a geometric monodromy homeomorphism
    $T : \mathcal{F}_{f,a} \to \mathcal{F}_{f,a}$ which in turn induces a unique algebraic monodromy operator
    $$T^* : H^*(\mathcal{F}_{f,a};\mathbb{C}) \to H^*(\mathcal{F}_{f,a};\mathbb{C})$$.
\end{thm}

\begin{conj}[Monodromy conjecture, weak version]
    Let $s_0$ be a pole of $Z_{top,0}(f;s)$. Then $e^{2i\pi s_0}$ is an eigenvalue of the monodromy $T_{x_0} : H^*(\mathcal{F}_{f,x_0};\mathbb{C}) \to H^*(\mathcal{F}_{f,x_0};\mathbb{C})$ for some $x_0 \in \{f=0\}$ in a neighborhood of the origin.
\end{conj}

The strong version of the conjecture predicts that every pole of $Z_{top,0}(f;s)$, or even more strongly, that every pole of $Z_{mot}(f;\mathbb{L}^{-s})$ is a root of the Bernstein-Sato polynomial $b_{f,0}(s)$ of $f$. By the work of Malgrange and Kashiwara, it is known that every root of $b_{f,0}(s)$ induces an eigenvalue of the monodromy.

Although the monodromy conjecture, even in its weak version, is still widely open, it has been proved in certain special cases, such as:\\
$\bullet$ the case of curves \cite{loes3},\\
$\bullet$ the case of homogeneous surfaces \cite{bartolo2},\\
$\bullet$ the case of hyperplane arrangements \cite{budur},\\
$\bullet$ the case of Newton-non-degenerate hypersurfaces singularities of four variables \cite{lemahieu1}.\\

For curves, Veys provides in \cite{veys1} the following criterion to filter out the true poles among the set of candidate poles in the resolution graph. More precisely, take $f\in \mathbb{C}[x,y] $, let $\sigma : (X,\sigma^{-1}(0)) \to (\mathbb{C}^2,0)$ be the \textit{canonical embedded resolution} of $f$ and denote by $\underset{j \in J}{\bigcup} E_j$ the decomposition of $\sigma^{-1}(f^{-1}(0))$ into irreducible components.

\begin{thm}[\cite{veys1} Theorem 4.3]
\label{thm 5}
Let $s_0 \in \mathbb{Q}$. Then $s_0$ is a pole of $Z_{top,0}(f;s)$ if and only if $s_0=-\frac{\nu_i}{N_i}$ for an exceptional curve $E_i$ intersecting at least 3 times other components or $s_0= -\frac{1}{N_i}$ for an irreducible component $E_i$ of the strict transform of $f$. The result also holds for the motivic zeta function $Z_{mot,0}(f;\mathbb{L}^{-s})$.
\end{thm}

On the other hand, zeta functions of motivic type have also been studied in real geometry for example in \cite{koike1}, \cite{goulw1}, \cite{goulw2}, and \cite{campesato}. For example, with the aim of obtaining invariants of the blow-Nash equivalence, Fichou uses the virtual Poincaré polynomial \cite{mccrory1} to define a zeta function $Z(f;T) \in \mathbb{Z}[u,u^{-1}] [[T]]$ and zeta functions with signs $Z^{\pm}(f;T)$ associated to a Nash function germ $f :( \mathbb{R}^d,0) \to (\mathbb{R},0)$ \cite{goulw1}. For these zeta functions, one also has a Denef-Loeser type formula expressing the fact that these functions are rational i.e. they belong to $\mathbb{Z}[u,u^{-1}](T)$. In particular, it still makes sense to study the poles of zeta functions in this setting.\\ 
In \cite{campesato} and \cite{goulw1}, the authors work in the Nash framework, which forces the associated motivic zeta functions to have coefficients in the Grothendieck ring K$_0(\mathcal{AS})$ of arc-symmetric sets \cite{kurdyka}. For simplicity, here we will only consider polynomial functions $f \in \mathbb{R}[x_1,\dots,x_d]$ and therefore remain in the category of real algebraic varieties.

This paper is organized as follows. In the first section, we review the construction of real motivic zeta functions in the algebraic setting. We then define real topological zeta functions $Z_{top,0}(f;s), Z_{top,0}^{\pm}(f;s) \in \mathbb{Q}(s)$, which are sometimes more suitable for the study of poles and constitute the real analogue of the complex topological zeta function introduced above. These real topological zeta functions are specializations of the real motivic zeta functions and are therefore invariants of the blow-Nash equivalence.

In the second section, we provide a complete description of the poles of real topological and motivic naive zeta functions for curves. To do this, we follow the approach of \cite{veys1} and adapt Veys' arguments to the real setting. More precisely, we first study the contribution of a component for a given candidate pole and, using the real dual graph of the resolution, show that different contributions do not cancel each other out. This allows us to establish theorem \ref{thm 4}, which provides a numerical criterion to filter out the true poles from the resolution graph. In particular we prove the following.

\begin{thm}[= Theorem \ref{thm 4}]
    Let $f\in \mathbb{R}[x,y]$, $\sigma :(X, \sigma^{-1}(0)) \to (\mathbb{A}^2_{\mathbb{R}},0)$ the canonical embedded resolution of $f$ and $Z_{top,0}(f;s)$ denote the real topological zeta function associated with $f$. Then $s_0 \in \mathbb{Q}$ is a pole of $Z_{top,0}(f;s)$ if and only if $s_0=-\frac{1}{N_i}$ for some irreducible component of the strict transform $E_i$ such that $E_i(\mathbb{R}) \neq \emptyset$ or $s_0 = -\frac{\nu_i}{N_i}$ for some exceptional curve $E_i$ satisfying $(E_i \cdot \sum\limits_{j \neq i} E_j) \geq 3$. Equivalently, one has 
    $$\text{Poles}(Z_{top,0}(f;s)) = \text{Poles}(Z_{top,0}(f_{\mathbb{C}};s)) \cap  \{~ -\frac{\nu_i}{N_i} \mid  i \in J_{\mathbb{R}}~\}$$
    where $J_{\mathbb{R}}$ denote the sets of components of $(f \circ \sigma)^{-1}(0)=\underset{j \in J}{\bigcup} E_j$ whose real locus is non-empty.
\end{thm}

In the third section, we study the poles of real topological and motivic zeta functions with signs in the case of curves. The approach is the same as in the second section, except that the computation of a contribution for a given pole candidate is more intricate. In particular, we show (Corollary \ref{coro 5}) that 
$$\text{Poles}(Z_{top,0}^{\pm}(f;s)) \subset \text{Poles}(Z_{top,0}(f;s))  \cap \{~ -\frac{\nu_i}{N_i} \mid  i \in J_{\mathbb{R}}^{\pm}~\}.$$
 In the case with signs, cancellations can occur and the poles of topological zeta functions and motivic zeta functions do not always coincide (see Example \ref{ex 6}), unlike in the naive case.  To conduct a more precise study of poles, we therefore examine the contributions at the level of the virtual Poincaré polynomial. In particular we prove the following.

 \begin{thm} [= Corollary \ref{cor 6}]
     Let $s_0 \in \mathbb{Q}$. Assume there exists exactly one $i \in J_{\mathbb{R}}^{\pm}$ such that either $E_i$ is an exceptional curve satisfying $(E_i\cdot \sum\limits_{j\neq i } E_j) \geq 3$ and $s_0=-\frac{\nu_i}{N_i}$ or such that $E_i$ is an irreducible component of the strict transform and $s_0=-\frac{1}{N_i}$. Then $s_0$ is a pole of $Z_{\beta,0}^{\pm}(f;u^{-s})$.\\
    In cases where there is at most one nonzero contribution for every candidate pole as above, this yields
    $$\text{Poles}(Z_{\beta,0}^{\pm}(f;u^{-s}))= \text{Poles}(Z_{\beta,0}(f;u^{-s}))  \cap \{~ -\frac{\nu_i}{N_i} \mid  i \in J_{\mathbb{R}}^{\pm}~\}.$$
 \end{thm}

Finally, we propose an interpretation of the poles of these real zeta functions in terms of eigenvalues of the monodromy acting on the Milnor fiber at a point in the real locus of $f$ close to the origin.

\bigskip
\noindent
{\bf Acknowledgements.} This work is part of the author's PhD thesis. He thanks his advisor Goulwen Fichou for his suggestions and support during the preparation of this work.


\bigskip
\noindent

\section{Real motivic zeta functions}
\label{section 1}

An algebraic variety over $\mathbb{R}$ will refer to a reduced scheme of finite type over $\mathbb{R}$, while the term real algebraic variety will be reserved for varieties as defined in \cite{bcr}. If $X$ is an algebraic variety over $\mathbb{R}$, the set $X(\mathbb{R})$ of real closed points of $X$ is naturally endowed with a structure of real algebraic variety and, a fortiori, with a structure of real analytic space. Conversely, if $Y$ is a real algebraic variety, there exists an algebraic variety $X$ over $\mathbb{R}$ (generally not unique) such that $X(\mathbb{R})$ and $Y$ are isomorphic as real algebraic varieties. However, the category of algebraic varieties over $\mathbb{R}$ is not equivalent to that of real algebraic varieties since the latter has strictly more morphisms. This larger class of morphisms implies, for example, that every quasi-projective real algebraic variety is affine (\cite{bcr} Theorem 3.4.4). Finally, if $X$ is an algebraic variety over $\mathbb{R}$, we will denote  $X(\mathbb{C})$ the set of complex points of its complexification.

\begin{defi}
    \label{def 1}
    We denote by K$_0(\text{Var}_{\mathbb{R}})$ (resp. K$_0$($\mathbb{R}$Var)) the free abelian group generated by the isomorphism classes $[X]$ of algebraic varieties over $\mathbb{R}$ (resp. of real algebraic varieties) modulo all relations of the form $[X]-[Y]- [X\setminus Y]$ whenever $Y$ is a closed subvariety of $X$. We equip K$_0(\text{Var}_{\mathbb{R}})$ and K$_0$($\mathbb{R}$Var) with a ring structure by setting $[X]\times [Y] = [X \times Y]$ and we refer to K$_0(\text{Var}_{\mathbb{R}})$ (resp. K$_0$($\mathbb{R}$Var)) as the Grothendieck ring of algebraic varieties over $\mathbb{R}$ (resp. the Grothendieck ring of real algebraic varieties). There is a natural ring morphism K$_0(\text{Var}_{\mathbb{R}}) \to  \text{K}_0(\mathbb{R}$Var) that maps $[X]$ to $[X(\mathbb{R})]$.
    Finally, we denote by $\mathcal{M}_{\mathbb{R}}$ the localization of K$_0$($\mathbb{R}$Var) with respect to the class of the affine line $\mathbb{L}=[\mathbb{R}]$.
\end{defi}

\begin{defi}
For $n \geq 1$, we denote by $\mathcal{L}_n(\mathbb{R}^d,0)$ the space of formal arcs truncated at order $n$ and starting at the origin in $\mathbb{R}^d$, i.e.

$$ \mathcal{L}_n(\mathbb{R}^d,0) = \{~ \gamma :\mathbb{R} \to \mathbb{R}^d \text{  formal arcs } \mid \gamma(0)=0 ~\}/ \sim $$
where $\gamma_1 \sim \gamma_2$ if and only if $\gamma_1(t) = \gamma_2(t)$ (mod $t^{n+1}$). In other words 
$$\mathcal{L}_n(\mathbb{R}^d,0) = \{~ \gamma(t) = \sum\limits_{i=1}^n a_i t^{i}\mid (a_1,\dots,a_n) \in \mathbb{R}^{nd} ~\} \simeq \mathbb{R}^{nd}.$$

\end{defi}

We now consider $f\in \mathbb{R}[x_1,\dots,x_d]$ vanishing at the origin, and we denote by $f_{\mathbb{C}}$ its complexification, that is, the polynomial $f$ viewed as an element of $\mathbb{C}[x_1,\dots,x_d]$.

\begin{defi}
    \label{def 2}
    We denote by $\mathcal{X}_n(f)$ the set of all arcs $\gamma \in \mathcal{L}_n(\mathbb{R}^d,0)$ whose order of contact with the hypersurface $\{ f =0\}$ equals $n$, i.e.
    $$\mathcal{X}_n(f)=\{~ \gamma \in \mathcal{L}_n(\mathbb{R}^d,0) \mid\text{ord}_t(f\circ\gamma) = n ~\} $$
    and by $\mathcal{X}_n^{\pm}(f) \subset \mathcal{X}_n(f)$ the subsets
    $$\mathcal{X}_n^{\pm}(f)=\{~ \gamma \in \mathcal{X}_n(f) \mid (f\circ \gamma) (t) = \pm t^n +o(t^{n+1}) ~\}.$$
    
    The sets $\mathcal{X}_n(f)$ and $\mathcal{X}_n^{\pm}(f)$ are Zariski-constructible (i.e. unions of locally closed subsets) in $\mathcal{L}_n(\mathbb{R}^d,0)\simeq \mathbb{R}^{nd}$. In particular, $\mathcal{X}_n(f)$ and $\mathcal{X}_n^{\pm}(f)$ defines elements of K$_0(\mathbb{R}\text{Var})$. These constructible sets refine the Fukui invariants (resp. the  Fukui invariants with signs) which have been studied in \cite{koike2} and \cite{fukui1}.
\end{defi}

\begin{rem}
    \label{rem 1}
    Let us emphasize that in the ring K$_0(\mathbb{R}\text{Var})$ we consider the real points of algebraic varieties over $\mathbb{R}$. For example, if $f=x^2+y^2 $ one has 
    $$\mathcal{X}_2(f) \simeq \{~ (a_1,a_2,b_1,b_2) \in \mathbb{R}^4 \mid a_1^2+b_1^2 \neq 0 ~\}= \mathbb{R}^2\times\mathbb{R}^2\setminus\{0\}$$ and therefore 
    $[\mathcal{X}_2(f)] = \mathbb{L}^2(\mathbb{L}^2-1) \text{ in } \text{K}_0(\mathbb{R}\text{Var})$. For the complexification $f_{\mathbb{C}}$, one has 
    $$\mathcal{X}_2(f_{\mathbb{C}})= \simeq \{~ (a_1,a_2,b_1,b_2) \in \mathbb{C}^4 \mid a_1^2+b_1^2 \neq 0 ~\} \simeq \mathbb{C}^2\times\mathbb{C}^2\setminus \{xy=0\} $$ so that 
    $\mathcal{X}_2(f_{\mathbb{C}}) = \mathbb{L}^2(\mathbb{L}^2-2\mathbb{L}+1)$ in K$_0(\text{Var}_{\mathbb{C}})$.
\end{rem}

\begin{defi}
    \label{def 3}
    The real local and naive motivic zeta function associated with $f$ is the formal power series 
    $$Z_{mot,0}(f;T) = \sum\limits_{n \geq 1} [\mathcal{X}_n(f)] \mathbb{L}^{-nd}T^n \in \mathcal{M}_{\mathbb{R}}[[T]].$$
    It is therefore an invariant (in the sense of \ref{thm 1}) of the real points of $\{f=0\}$ in a neighborhood of the origin.
    The motivic zeta functions with signs are defined by 
    $$Z_{mot,0}^{\pm}(f;T) = \sum\limits_{n \geq 1} [\mathcal{X}_n^{\pm}(f)] \mathbb{L}^{-nd}T^n \in \mathcal{M}_{\mathbb{R}}[[T]].$$
\end{defi}

To obtain more concrete invariants, it is useful to have motivic measures of real algebraic varieties, i.e., ring morphisms K$_0(\mathbb{R}\text{Var}) \to A$. The finest additive and multiplicative invariant of real algebraic varieties known to date is the virtual Poincaré polynomial.

\begin{thm}[\cite{mccrory1} Corollary 2.2]
    There exists a unique ring morphism $ \beta : \text{K}_0(\mathbb{R}\text{Var}) \to \mathbb{Z}[u]$ such that $\beta(X) = \sum\limits_i \text{dim}(\text{H}_i(X;\mathbb{Z}/2\mathbb{Z})) u^{i}$ when $X$ is smooth and compact. One can recover the Euler characteristic with compact support of $X$ by evaluating $\beta(X)$ at $u=-1$.
\end{thm}

\begin{ex}
\label{ex 5}
    The real projective line $\mathbb{P}^1(\mathbb{R})$ is smooth and compact, hence $\beta(\mathbb{P}^1(\mathbb{R}))=1+u$. By additivity one can deduce, for example, that $\beta(\mathbb{P}^1(\mathbb{R}) \setminus \{ \text{k points} \})=u+1-k$. Note that the virtual Poincaré polynomial \textit{is not a topological invariant} of real algebraic varieties (see \cite{mccrory1} Example 2.7). 
\end{ex}

\begin{rem}

\label{rem 2}
    As soon as one has a ring morphism $ \varphi :  \text{K}_0(\mathbb{R}\text{Var}) \to A$ with $\varphi(\mathbb{L}) \neq 0$, it induces a morphism  $\mathcal{M}_{\mathbb{R}}[[T]] \to A[\varphi(\mathbb{L})^{-1}] [[T]]$ and thus a specialization of the motivic zeta functions associated with $f$. For example, if $\chi_c : \text{K}_0(\mathbb{R}\text{Var}) \to \mathbb{Z}$ denotes the Euler characteristic with compact support, one obtains the specialization
    $$Z_{\chi_c, 0}(f;T) = \sum\limits_{n \geq 1} \chi_{c}(\mathcal{X}_n(f)) (-1)^{-nd}T^n \in \mathbb{Z}[[T]]$$
    which has been studied in \cite{koike1}. If $\beta: \text{K}_0(\mathbb{R}\text{Var}) \to \mathbb{Z}[u]$ denotes the virtual Poincaré polynomial, one obtains the specialization 
    $$Z_{\beta,0}(f;T) = \sum\limits_{n \geq 1} \beta(\mathcal{X}_n(f)) u^{-nd}T^n \in \mathbb{Z}[u,u^{-1}][[T]]$$
    which was studied in \cite{goulw1}.
\end{rem}

\begin{notation}
    \label{not 1}
Let $ \sigma : (X,\sigma^{-1}(0)) \to (\mathbb{R}^d,0)$ be an algebraic modification (i.e., a proper and birational map that is an isomorphism outside the zero locus of $f$) such that the divisors $\sigma^{*} (\text{div}(f)) = \text{div}(f \circ \sigma)$ and $\sigma^{*}(dx_1 \wedge \dots \wedge dx_d)$ are simultaneously normal crossings.\\
This means that for every $p \in \sigma^{-1}(f^{-1}(0))$ there is a local coordinate system $(y_1,\dots,y_d)$ centered at $p$ in $X$ such that $f( \sigma(y_1,\dots,y_d)) = u y_1^{N_1}\dots y_d^{N_d}$ with  $u(p)\neq 0$ and also $\text{jac }\sigma = v y_1^{\nu_1 - 1}\dots y_d^{\nu_d - 1}$ with $v(p) \neq 0$. Such a modification always exists according to Hironaka's theorem \cite{hiro1} and we will also say that $\sigma :(X, \sigma^{-1}(0)) \to (\mathbb{R}^d,0)$ is an embedded resolution of $f$.\\
Once we have an embedded resolution, we denote by $\underset{j \in J}{\bigcup} E_j$ the decomposition into irreducible components of $\sigma^{-1}(f^{-1}(0))$, where, by definition, the $E_j$ are smooth irreducible hypersurfaces in $X$ that intersect transversally. We then denote by $\underset{I}{ \bigsqcup}E^{0}_I$ the canonical stratification of $\underset{j \in J}{\bigcup} E_j$ into smooth subvarieties, where $I$ runs over the set of non-empty subsets of $J$ and $E^{0}_I = \underset{i \in I}{\bigcap} E_i \setminus \underset{j \in J\setminus I}{\bigcup} E_j$. Finally, the numerical data of the resolution  are the integers $N_i = \text{mult}_{E_i} (f\circ \sigma)$ and $\nu_i =1+ \text{mult}_{E_i} ( \text{jac } \sigma)$, which can be computed in a local coordinate system as above and do not depend on the chosen coordinate. In other words, one has 
$$\text{div}(f \circ \sigma) = \sum\limits_{j \in J} N_j E_j \text{ and } K_X = \sum\limits_{j \in J} (\nu_j-1)E_j.$$
Note that if $E_i$ is an irreducible component of the strict transform, then $\nu_i=1$, and when $f$ is reduced, we also have $N_i=1$.

\end{notation}

There exists a Denef-Loeser type formula that expresses the real motivic zeta function of $f$ in terms of an embedded resolution as described above.

\begin{thm}[\cite{goulw1} Proposition 4.2.]
\label{thm 2}
Let $\sigma : (X,\sigma^{-1}(0)) \to (\mathbb{R}^d,0)$ be an embedded resolution of $f$. Using the notation from \ref{not 1}, one has 
$$Z_{mot,0}(f;T) = \sum\limits_{ \emptyset \neq I \subset J} (\mathbb{L}-1)^ {|I|} [E^{0}_I \cap \sigma^{-1}(0)] \underset{i \in I}{\prod}\frac{\mathbb{L}^{-\nu_i}T^{N_i}}{1-\mathbb{L}^{-\nu_i}T^{N_i}}$$
\end{thm}

\begin{rem}
\label{rem 9}
     The motivic zeta function of $f$ is therefore a rational function; more precisely, $Z_{mot,0}(f;T)$ belongs to $\mathcal{M}_{\mathbb{R}} [T][\frac{1}{1-\mathbb{L}^{-\nu_i}T^{N_i}}]_{i \in I}.$ In particular, the sequence of $[\mathcal{X}_n(f)]$ is determined by a finite amount of data.\\
    On the other hand, this also proves that the above expression does not depend on the chosen resolution, since $Z_{mot,0}(f;T)$ was defined intrinsically in \ref{def 3}.
\end{rem}

To express the rationality of zeta functions with signs, we must introduce coverings $\widetilde{E}_I^{0,\pm}$ of the strata $E_I^0$ as follows. Every point of $E_I^0$ has an open affine neighborhood $U$ on which $f\circ \sigma = u \prod\limits_{i \in I} y_i^{N_i}$ where $u$ is a unit. We then define the sets 
$$R_U^{\pm} = \{ ~(x,t) \in (E_I^0 \cap U)\times\mathbb{R} \mid t^mu(x) = \pm 1 ~\}$$
where $m=\text{gcd}(N_i)$. The variety $\widetilde{E}_I^{0,\pm}$ is obtained by gluing the $R_U^{\pm}$ along the open sets $E_I^0 \cap U$. One then has a covering $\widetilde{E}_I^{0,\pm} \to E_I^0$ which is locally trivial for the Euclidean topology.
For simplicity, we denote by $\widetilde{E}_I^{0,\pm} \cap \sigma^{-1}(0)$ the restriction of the covering $\widetilde{E}_I^{0,\pm} \to E_I^0$ above $E_I^0 \cap \sigma^{-1}(0)$.

\begin{thm}[\cite{goulw1} Proposition 4.4]
    Let $ \sigma : (X,\sigma^{-1}(0)) \to (\mathbb{R}^d,0)$ be an embedded resolution of $f$. One has
    $$Z_{mot,0}^{\pm}(f;T) = \sum\limits_{ \emptyset \neq I \subset J} (\mathbb{L}-1)^{|I|-1} [\widetilde{E}_I^{0,\pm} \cap \sigma^{-1}(0)] \underset{i \in I}{\prod}\frac{\mathbb{L}^{-\nu_i}T^{N_i}}{1-\mathbb{L}^{-\nu_i}T^{N_i}}.$$
\end{thm}

One of the main motivations for studying these zeta functions in real geometry is that they provide invariants for the classification of Nash function germs under blow-Nash equivalence.

\begin{thm}[\cite{goulw1} Theorem 4.9.]
    \label{thm 1}
    The functions $Z_{\beta,0}(f;T)$ and $Z_{\beta,0}^{\pm}(f;T)$ are invariants for the blow-Nash equivalence. 
\end{thm}

For instance, the zeta functions $Z_{\beta,0}$ and $Z_{\beta,0}^{\pm}$ defined via the virtual Poincaré polynomial are used in \cite {goulw1} to classify Brieskorn polynomials in two variables and in \cite{goulw3} to classify simple singularity germs for the blow-Nash equivalence. In a similar way, motivic zeta functions are used in \cite{campesato2} to classify Brieskorn polynomials in any variables up to blow-Nash equivalence (which coincide with the arc-analytic equivalence).

\begin{rem}
    \label{rem 3}
    From now on, we always begin by taking an embedded resolution $\sigma:(X,\sigma^{-1}(0)) \to (\mathbb{A}^{d}_{\mathbb{R}},0)$ in the schematic sense. We can then work in the category of real algebraic varieties by considering the induced morphism on real points $ (X(\mathbb{R}),\sigma^{-1}(0)(\mathbb{R})) \to (\mathbb{R}^d,0) $. To emphasize this, we will write
    $$Z_{mot,0}(f;T) = \sum\limits_{\emptyset \neq I \subset J} (\mathbb{L}-1)^ {|I|} [E^{0}_I \cap \sigma^{-1}(0)(\mathbb{R})] \underset{i \in I}{\prod}\frac{\mathbb{L}^{-\nu_i}T^{N_i}}{1-\mathbb{L}^{-\nu_i}T^{N_i}}.$$
    This schematic viewpoint will be necessary later when computing intersection numbers of real algebraic curves, taking into account both real and complex points.
    Note that the motivic zeta function of the complexification $f_{\mathbb{C}}$ is then given by 
    $$Z_{mot,0}(f_{\mathbb{C}};T) = \sum\limits_{\emptyset \neq I \subset J} (\mathbb{L}-1)^ {|I|} [E^{0}_{I} \cap \sigma^{-1}(0)(\mathbb{C)}] \underset{i \in I}{\prod}\frac{\mathbb{L}^{-\nu_i}T^{N_i}}{1-\mathbb{L}^{-\nu_i}T^{N_i}}.$$
\end{rem}

\begin{defi}[Poles of zeta functions]
\label{def 6}
    Let us consider the zeta function $Z_{\beta,0}$ defined at the level of the virtual Poincaré polynomial. According to the rationality formula, one can write $Z_{\beta,0}(f;T)=\frac{P(T)}{Q(T)}$ where $P(T),Q(T) \in \mathbb{Z}[u,u^{-1}][T]$. The ring $\mathbb{Z}[u,u^{-1}]$ is a subring of $\underset{k \geq 1}{\bigcup}\mathbb{Z}[u^{\frac{1}{k}}, u^{-\frac{1}{k}}]$, so that for any $s \in \mathbb{Q}$, the zeta function $Z_{\beta,0} \in \underset{k \geq 1}{\bigcup}\mathbb{Z}[u^{\frac{1}{k}},u^{-\frac{1}{k}}][T] $ can be evaluated at $T=u^{-s}$, which is natural by analogy with Igusa zeta functions. This yields

     $$Z_{\beta,0}(f;u^{-s}) = \sum\limits_{\emptyset \neq I \subset J} (u-1)^ {|I|} \beta(E^{0}_I \cap \sigma^{-1}(0)(\mathbb{R}) \underset{i \in I}{\prod}\frac{u^{-(\nu_i+sN_i)}}{1-u^{-(\nu_i+sN_i)}}$$
     which can also be written as
     $$Z_{\beta,0}(f;u^{-s}) = \sum\limits_{\emptyset \neq I \subset J} \beta(E^{0}_I \cap \sigma^{-1}(0)(\mathbb{R}) \underset{i \in I}{\prod}\frac{u-1}{u^{(\nu_i+sN_i)}-1}.$$

    We say that $s_0 \in \mathbb{Q}$ is a pole of $Z_{\beta,0}(f;u^{-s})$ if and only if $u^{-s_0}$ is a pole of $Z_{\beta,0}(f;T)$. Equivalently, $s_0 \in\mathbb{Q}$ is a pole of $Z_{\beta,0}(f;u^{-s})$ if and only if there exist $P,Q \in \mathbb{Z}[u,u^{-1}] [T]$ such that $Z_{\beta,0}(f;T)=\frac{P(T)}{Q(T)}$ and such that $P(u^{-s_0}) \neq 0$ and $Q(u^{-s_0})=0$. 
    The definition of poles for motivic zeta functions is more subtle, due to the fact that it is not known whether the ring $\mathcal{M}_{\mathbb{R}}$ is a domain. For instance, it is known that the rings K$_0(\text{Var}_\mathbb{C})$ and K$_0(\mathbb{R}\text{Var})$ are not integral domains (see \cite{Poonen_2002}, \cite{goulw5}). For a precise definition of poles in this context, we refer to section 4 of \cite{rodrigues3}. We will simply note that, since $Z_{\beta,0}(f,u^{-s})$ is a specialization of $Z_{mot,0}(f,\mathbb{L}^{-s})$, any pole of $Z_{\beta,0}(f,u^{-s})$ is also a pole of $Z_{mot,0}(f,\mathbb{L}^{-s})$. The poles of zeta functions with signs are defined in a completely analogous way.
\end{defi}

Our study of the poles of real motivic zeta functions is motivated on the one hand by the fact that the set of these poles constitutes an invariant for the blow-Nash equivalence and, on the other hand, by the fact that in the complex setting, these poles have (at least conjecturally) a significant topological interpretation.

\begin{ex}
    Take $f=x^3+y^3$. The blowing-up at the origin gives an embedded resolution for $f$, and one finds
    $$Z_{\beta,0}(f;T) = u(u-1)\frac{u^{-2}T^3}{1-u^{-2} T^3} + (u-1)^2\frac{u^{-2}T^3}{1-u^{-2}T^3}\frac{u^{-1}T}{1-u^{-1}T}$$
    so the candidate poles are $-1$ and $-\frac{2}{3}$.
    After simplification, one has 
    $$Z_{\beta,0}(f;u^{-s}) = \frac{(u-1)u^{-(2+3s)}(u-u^{-(s+1)})}{(1-u^{-(2+3s)})(1-u^{-(s+1)})}.$$
    Evaluating the numerator at $s=-1$ and $s= -\frac{2}{3}$ gives $(u-1)^2u$ and $(u-1)(u-u^{-\frac{1}{3}})$ respectively, both of which are nonzero. Therefore, the poles of $Z_{\beta,0}(f;u^{-s})$ are indeed $-1$ and $-\frac{2}{3}$.
\end{ex}

Let $\sigma : (X,\sigma^{-1}(0)) \to (\mathbb{A}^d_{\mathbb{R}},0)$ be an embedded resolution of $f$. By analogy with the complex case, it is natural to associate to $f$ a so-called real topological zeta function $Z_{top,0}(f;s)$, which will be an element of $\mathbb{Q}(s)$. 

\begin{defi}
\label{def 4}
    We denote by $\mu : \text{K}_0(\mathbb{R}\text{Var}) \to \mathbb{Z}$ the additive invariant defined as the composition of the virtual Poincaré polynomial $\beta: \text{K}_0(\mathbb{R}\text{Var}) \to \mathbb{Z}[u]$
    with the evaluation map $\mathbb{Z}[u] \to \mathbb{Z} $ sending $u$ to $1$. 
\end{defi}

\begin{rem}
    Let us note that $\mu$ is neither the Euler characteristic, which is not an additive invariant of real algebraic varieties, nor the compactly supported Euler characteristic, which is the obtained by composing $\beta$ with the evaluation map at $-1$. In particular, $\mu$ is not a topological invariant.
\end{rem}

\begin{defi}
\label{def 8}
    We define the real local topological zeta function of $f$ by
    $$Z_{top,0}(f;s)= \sum\limits_{ \emptyset \neq I \subset J} \mu(E^{0}_I \cap \sigma^{-1}(0)(\mathbb{R})) \underset{i \in I}{\prod}\frac{1}{\nu_i+sN_i} \in \mathbb{Q}(s).$$
    The topological zeta functions with signs are defined by
    $$Z_{top,0}^{\pm}(f;s)= \sum\limits_{ \emptyset \neq I \subset J} \mu(\widetilde{E}_I^{0,\pm} \cap \sigma^{-1}(0)(\mathbb{R})) \underset{i \in I}{\prod}\frac{1}{\nu_i+sN_i} .$$
\end{defi}

\begin{rem}
\label{rem 5}

    The real topological zeta function can be defined as 
    $$Z_{top,0}(f;s) = \lim_{u\to 1} Z_{\beta,0}(f;u^{-s})$$ where a first-order expansion in the expression of \ref{def 6} shows that terms of the form $\frac{u-1}{u^{\nu_i +sN_i}-1}$ tend to $\frac{1}{\nu_i +sN_i}$ as $u$ approaches $1$. In particular, $Z_{top,0}(f;s)$ is a specialization of the zeta function defined at the level of the virtual Poincaré polynomial. It follows that $Z_{top,0}(f;s)$ satisfies the following properties: \\
    $\bullet$ The function $Z_{top,0}(f;s)$ is well defined, i.e., it does not depend on the chosen resolution, since $Z_{\beta,0}(f;u^{-s})$ is defined intrinsically as in \ref{def 3}.\\
    $\bullet$ The function $Z_{top,0}(f;s)$ is an invariant of the blow-Nash equivalence.\\
    $\bullet$ One has the following inclusions
    $$\text{Poles}(Z_{top,0}(f;s)) \subset \text{Poles}(Z_{\beta,0}(f;u^{-s})) \subset \text{Poles}(Z_{mot,0}(f;\mathbb{L}^{-s})).$$
    
    According to \ref{def 6}, the poles of $Z_{\beta,0}^{\pm}(f;u^{-s})$ are the same as the poles of $(u-1)Z_{\beta,0}^{\pm}(f;u^{-s})$. We can also define the topological zeta functions with signs as
    $$Z_{top,0}^{\pm}(f;s) = \underset{u\to 1}{\text{lim } }(u-1) Z_{\beta,0}^{\pm}(f;u^{-s})$$
    and it follows that the properties mentioned above also hold for topological zeta functions with signs. We will see that the above inclusions for the poles of naive zeta function are in fact equalities in the case of curves whereas the same inclusions can be strict for zeta function with signs, even in the case of curves.
    Finally, let us mention that, as in the complex case, the name “topological zeta function” may not be entirely appropriate, since $Z_{top,0}$ is not a topological invariant. Moreover, these functions are different from the topological zeta functions studied by Koike and Parusiński in \cite{koike1}.

\end{rem}

\begin{ex}
\label{ex 4}
    Take $f=y^2-x^3$. By performing three successive blowings-ups, one obtains an embedded resolution of $f$, whose resolution graph with numerical data $E_i(\nu_i,N_i)$ is shown below.
    \begin{figure}[H]
        \centering
        \includegraphics[scale=0.35]{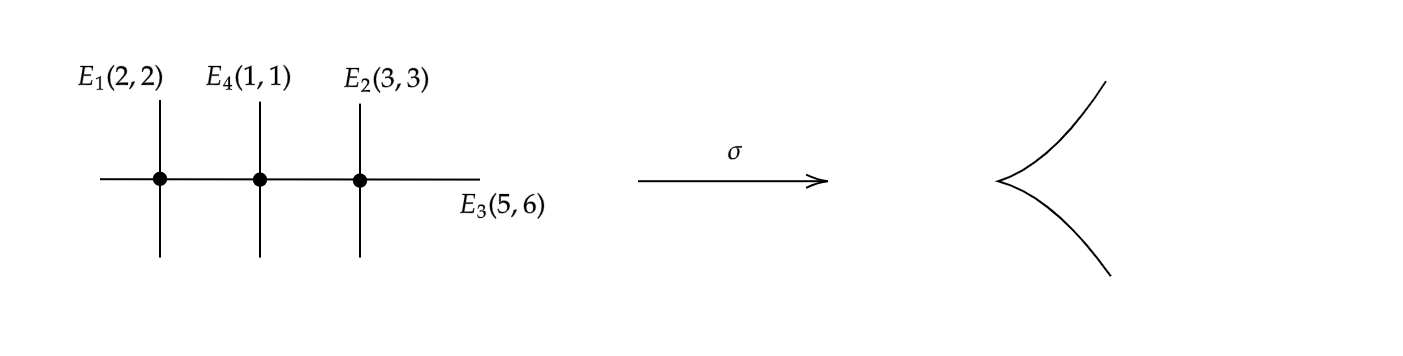}
    \end{figure}
    It follows that 
    $$Z_{top,0}(f;s) = \frac{1}{2+2s} + \frac{1} {3+3s} - \frac{1}{5+6s} + \frac{1}{(2+2s)(5+6s)} + \frac{1}{(3+3s)(5+6s)} + \frac{1}{(1+s)(5+6s)}$$
    and after simplification one has
    $$Z_{top,0}(f;s) = \frac{5+4s}{(s+1)(5+6s)}.$$
    Therefore, the poles of $Z_{top,0}(f;s)$ are $-1$ and $-\frac{5}{6}$.
\end{ex}

\begin{rem}
    
    More generally, if $f :(\mathbb{R}^d,0) \to (\mathbb{R},0)$ is a Nash function germ, one can also associate to $f$ a real topological zeta function, as well as topological zeta functions with signs. Indeed, by Corollary $2.4$ of \cite{goulw1}, there exists a unique morphism $\beta :\text{K}_0(\mathcal{AS}) \to \mathbb{Z}[u] $ that extends the virtual Poincaré polynomial to the Grothendieck ring of arc-symmetric sets \cite{kurdyka}. So one also has zeta functions $Z_{\beta,0}(f;T), ~Z_{\beta,0}^{\pm}(f;T) \in \mathbb{Z}[u,u^{-1}](T)$, and one can define similarly 
$$Z_{top,0}(f;s) = \lim_{u\to 1} Z_{\beta,0}(f;u^{-s}) ~~~\text{  and  } ~~~Z_{top,0}^{\pm}(f;s) = \underset{u\to 1}{\text{lim } }(u-1) Z_{\beta,0}^{\pm}(f;u^{-s}).$$
As in \ref{def 8}, one can also express or define $Z_{top,0}(f;s)$ and $Z_{top,0}^{\pm}(f;s)$ in terms of a Nash modification $\sigma :(X,\sigma^{-1}(0)) \to (\mathbb{R}^d,0)$ such that $\sigma^*(\text{div}(f))$ and $\sigma^*(dx_1 \wedge \dots \wedge dx_d)$ are simultaneously normal crossings.

\end{rem}

\section{Poles of the real naive zeta function for curves}

\label{section 2}

In this section, we provide a complete description of the poles of real naive zeta functions when $f \in \mathbb{R}[x,y]$, as is done in \cite{veys1} for the topological zeta function in the complex case. The goal is to give an analogous numerical criterion to identify the true poles from the resolution graph of the \textit{canonical embedded resolution} of $f$.\\
We will first study the contribution of a single component $E_i$ for a given candidate pole. Then, using the real total dual graph of the resolution, we will show that the nonzero contributions coming from different components do not cancel each other out. This section essentially involves adapting Veys' results to the real setting.

\begin{rem}
    \label{rem 4} For now take $f\in \mathbb{R}[x_1,\dots,x_d]$ and $\sigma : (X,\sigma^{-1}(0)) \to (\mathbb{A}^d_{\mathbb{R}},0)$ an embedded resolution of $f$. Let $J_{\mathbb{R}}$ denote the set of components whose real locus is non-empty, i.e.,
    $$J_{\mathbb{R}} = \{~ j \in J \mid E_{j}\ (\mathbb{R}) \neq \emptyset ~\}$$
    and let $I \subset J$. Assume that there exists $i \in I$ such that $i \notin J_{\mathbb{R}}$. Then
    $$E_I^0 \cap \sigma^{-1}(0)(\mathbb{R}) \subset E_i(\mathbb{R}) = \emptyset$$ 
    so $\beta(E_I^0 \cap \sigma^{-1}(0)(\mathbb{R})) = 0$ and a fortiori $\mu(E_I^0 \cap \sigma^{-1}(0)(\mathbb{R})) = 0$. In other words, we can also write 
    $$Z_{top,0}(f;s)= \sum\limits_{ \emptyset \neq I \subset J_{\mathbb{R}}} \mu(E^{0}_I \cap \sigma^{-1}(0)(\mathbb{R})) \underset{i \in I}{\prod}\frac{1}{\nu_i+sN_i}.$$
    This implies that the set of candidate poles of $Z_{top,0}(f;s)$, and thus the set of candidate poles of $Z_{mot,0}(f;\mathbb{L}^{-s})$), is  
    $$\{~ -\frac{\nu_i}{N_i} \mid  i \in J_{\mathbb{R}}~\}.$$
    Note also that the set of candidate poles of $Z_{mot,0}(f;\mathbb{L}^{-s})$ is always included in the set $\{~ -\frac{\nu_i}{N_i} \mid  i \in J ~\} $  of candidate poles of $Z_{mot,0}(f_{\mathbb{C}};\mathbb{L}^{-s})$.
\end{rem}

\begin{ex}
\label{ex 1}\begin{enumerate}
    \item The set of poles of $Z_{top,0}(f;s)$ is generally different from the set of poles of $Z_{top,0}(f_{\mathbb{C}};s)$. For example, if $f=x^{2k}+y^{2k}$ with $k \geq 2$, blowing up the origin gives an embedded resolution of $f$ where the strict transform of $f$ has no real points. Therefore, 
    $$Z_{top,0}(f;s)=\frac{2}{2+2ks}$$ and the unique pole of $Z_{top,0}(f;s)$ is $-\frac{1}{k}$. On the other hand, for the complexification $f_{\mathbb{C}}$, the strict transform is the union of $2k$ complex lines, so that
    $$Z_{top,0 }(f_{\mathbb{C}};s)= \frac{2-2k}{2+2ks}+\frac{2k}{(1+s)(2+2ks)} = \frac{2+2s-2ks}{(s+1)(2+2ks)}$$ 
    which has poles $-1$ and $-\frac{1}{k}$. Note also that $-\frac{1}{k}$ is a common pole of $Z_{top,0}(f;s)$ and $Z_{top,0}(f_{\mathbb{C}};s)$, but that the residues of these two functions at this pole are not equal.
    \item It may happen that $Z_{top,0}(f;s)$ and $Z_{top,0}(f_{\mathbb{C}};s)$ have a pole in common, but that the order of this pole differs between the two functions. For example, when $f=x^2+y^2$, one has
    $$Z_{top}(f;s)=\frac{2}{2+2s}=\frac{1}{1+s}$$
    while 
    $$Z_{top}(f_{\mathbb{C}};s)=\frac{2}{(2+2s)(1+s)}=\frac{1}{(1+s)^2}.$$
    \end{enumerate}
    
\end{ex}

\subsection{Study of a contribution}
\label{subsection 2.1}

 \text{ }\\
Let us briefly recall how any smooth algebraic surface $X$ defined over $\mathbb{R}$ can be equipped with a bilinear intersection form. We denote by Pic$(X)$ the Picard group of $X$, which can be identified with the group of divisors modulo linear equivalence, that is, modulo principal divisors.\\
If $C, C' \subset X$ are two algebraic curves defined over $\mathbb{R}$ and $p \in C \cap C'$, the intersection multiplicity of $C$ and $C'$ at $p$ is defined by
$$(C \cdot C')_p = \dim_{\mathbb{R}}  \frac{\mathcal{O}_{X, p}}{(g,h)}$$
where $g,h$ are local equations of $C$ and $C'$ in the neighborhood of $p$.
We will mainly use the fact that when $C$ and $C'$ intersect transversally at the point $p$, the intersection multiplicity $(C \cdot C')_p$ is equal to the dimension of the residue field of $p$ as a $\mathbb{R}$-vector space. In particular $(C \cdot C')_p = 1$ when $p$ is a real point and $(C \cdot C') = 2$ when $p$ is a complex point. If $C$ and $C'$ have no irreducible components in common, we then define the intersection number of $C$ and $C'$ by 
$$(C \cdot C') = \sum_{p \in C \cap C'} (C \cdot C')_p$$
By extension, this defines a symmetric bilinear intersection form 
\[\begin{array}[t]{lrcl}
 & \text{Pic}(X)\times \text{Pic}(X) & \longrightarrow & \mathbb{Z} \\
         & ([D],[D']) & \longmapsto & (D \cdot D')
\end{array}\]

From now on, we fix $f \in \mathbb{R}[x,y]$ and let $\sigma : (X,\sigma^{-1}(0)) \to (\mathbb{A}^2_{\mathbb{R}},0)$ be the canonical embedded resolution of $f$. As before, we denote by $\sum\limits_{j \in J} N_jE_j$ the principal divisor induced by $f \circ  
 \sigma$, where the $E_j$ are smooth irreducible curves on $X$ that intersect transversally. Let us note that $\sigma$ is the composition of a finite number of blowings-up, starting with the blowing-up of the origin, so that $\sigma^{-1}(0)$ is the union of the exceptional curves created by this sequence of blowings-up. Using the notation already introduced, we have seen that 
$$Z_{top,0}(f;s) = \sum\limits_{ \emptyset \neq I \subset J_{\mathbb{R}}} \mu(E^{0}_I \cap \sigma^{-1}(0)(\mathbb{R})) \underset{i \in I}{\prod}\frac{1}{\nu_i+sN_i}.$$
But since the $E_i$ are simultaneously normal crossings, we have $E_I^0 = \emptyset$ as soon as $|I| > 2$, so that
$$Z_{top,0}(f;s)= \sum\limits_{i \in J_{\mathbb{R}}} \frac{\mu(E^{0}_i \cap \sigma^{-1}(0)(\mathbb{R}))}{\nu_i+sN_i} +  \sum\limits_{\{i,j\} \subset J_{\mathbb{R}}} \frac{\mu(E_i(\mathbb{R})\cap E_j(\mathbb{R}))}{(\nu_i+sN_i)(\nu_j+sN_j)}.$$ 
Here $\beta(E_i(\mathbb{R})\cap E_j (\mathbb{R}))$, and a fortiori $\mu(E_i(\mathbb{R})\cap E_j (\mathbb{R}))$, is equal to the number of real intersection points of $E_i $ and $E_j$.

From the above expression, it already follows that any pole of $Z_{top,0}(f;s)$ has order at most 2.

\begin{prop}
\label{prop 1}
    Let $s_0 \in \mathbb{Q}$. Then $s_0$ is a pole of order $2$ of $Z_{top,0}(f;s)$ if and only if there exist distinct $i,j \in J_{\mathbb{R}}$ such that $E_i$ and $ E_j$ intersect at a real point and such that $s_0= -\frac{\nu_i}{N_i}=-\frac{\nu_j}{N_j}$.
\end{prop}

\begin{proof}
    According to the expression of $Z_{top}(f;s)$, one has 
    
    $$\underset{s\to s_0}{\text{lim}}(s-s_0)^2 Z_{top}(f;s) = \sum_{\{i,j\}\subset J_{\mathbb{R }}}   \underset{s\to s_0}{\text{lim}} (s-s_0)^2 \frac{\mu(E_i(\mathbb{R})\cap E_j(\mathbb{R}))}{(\nu_i+sN_i)(\nu_j+sN_j)}  $$ 
    If $s_0 \neq -\frac{\nu_i}{N_i}$ or $s_0 \neq -\frac{\nu_j}{N_j}$, then $\underset{s \to s_0}{\lim} (s-s_0)^2 \frac{\mu(E_i\cap E_j(\mathbb{R}))}{(\nu_i+sN_i)(\nu_j+sN_j)} = 0$. Now, when $i,j \in J$ are distinct and satisfy $s_0=-\frac{\nu_i}{N_i}=-\frac{\nu_j}{N_j}$, one has
    $$\underset{s\to s_0}{\text{lim}}(s-s_0)^2 \frac{\mu(E_i(\mathbb{R})\cap E_j(\mathbb{R}))}{(\nu_i+sN_i) (\nu_j+sN_j)} =\frac{\mu(E_i(\mathbb{R})\cap E_j(\mathbb{R}))}{N_iN_j} \geq 0$$
    where the inequality is strict if and only if $E_i$ and $E_j$ intersect at a real point. The result follows since $\underset{s\to s_0}{\text{lim}}(s-s_0)^2 Z_{top}(f;s)$ is the sum of these terms.
\end{proof}

\begin{rem}
\label{rem 11}
Let us consider two special cases within this remark. Assume that $f$ is already normal crossing, i.e., analytically equivalent to  $\lambda x^Ny^M$ for some $N,M \in \mathbb{N}$ and $\lambda \in \mathbb{R}^*$. Then
$$Z_{top,0}(f;s) = \frac{1}{(1+sN)(1+sM)} = Z_{top,0}(f_{\mathbb{C}};s)$$

Now assume that there exists an exceptional curve $E_i$ such that $(E_i \cdot \sum\limits_{j \neq i } E_j) = 2$ and such that $E_i$ does not intersect any component at a real point. This implies that the canonical embedded resolution of $f$ is obtained by a single blowing-up of the origin $\sigma : ( \text{Bl}_0 \mathbb{A}^2_{\mathbb{R}},E) \to (\mathbb{A}^2_{\mathbb{R}},0)$, which creates a unique exceptional curve $E_i =E$. Furthermore, the strict transform of $f$ is a smooth irreducible curve whose real locus is empty and which intersects $E$ at a complex point. On an affine chart $U \simeq \mathbb{A}^2_{\mathbb{R}}$ of $\text{Bl}_0 \mathbb{A}^2_{\mathbb{R}}$, one has $f(\sigma(x,y))= y^{N} u(x,y)$ 
with
$$E \cap U =\{~ y = 0 ~\} \simeq \mathbb{A}^1_{\mathbb{R}}$$
and $u(x,0) \in \mathbb{R}[x]$ does not vanish on $\mathbb{R}$ but has two complex conjugate roots. Therefore, one can write $u(x,0)=(ax^2+bx+c)^M$ where $b^2-4ac < 0$. We also know that the multiplicity of $E$ is equal to the algebraic multiplicity of $f$ at $0$, thus $N = 2 M$ and 
$$Z_{top,0}(f;s) = \frac{2}{2+Ns} = \frac{1}{1+sM}.$$
For the complexification $f_{\mathbb{C}}$, one finds
$$Z_{top,0}(f_\mathbb{C};s) = \frac{0}{2+Ns} + \frac{2}{(2+Ns)(1+Ms)} = \frac{1}{(1+sM)^2}$$
\end{rem}

Throughout the rest of this section, we will assume that $f$ is not as in the above remark. This means that for all $i\in J_{\mathbb{R}}$, one has $(E_i \cdot \sum\limits_{j \neq i } E_j) \geq 1$, and that if $(E_i \cdot \sum\limits_{j \neq i} E_j) =2$, then $E_i$ intersects the other components at two real points.

\begin{notation}
\label{nota 2}
Let us fix $i \in J_{\mathbb{R}}$ and study the contribution of a component $E_i$ to the residue of $Z_{top,0}(f;s)$ at the candidate pole $s_0=-\frac{\nu_i}{N_i}$. Assume that for all $j\in J_{\mathbb{R}}\setminus \{i\}$ we have $\frac{\nu_i}{N_i} \neq \frac{\nu_j}{N_j}$ as soon as $E_i$ and $E_j$ intersect at a real point, otherwise, we saw above that $s_0$ is a pole of order 2. We then truncate the function $Z_{top,0}(f;s)$ by keeping only the terms 
$$\frac{\mu(E_i^0 \cap \sigma^{-1}(0)(\mathbb{R}))}{\nu_i +sN_i} + \sum\limits_{j \neq i } \frac{\mu(E_i(\mathbb{R}) \cap E_j(\mathbb{R}))}{(\nu_i+sN_i) (\nu_j+sN_j)}.$$ 
We denote by $\mathcal{R}_{top,i}$ the residue of this expression at $s_0$. Therefore,
$$\mathcal{R}_{top,i} = \frac{1}{N_i}\left(\mu(E_i^0 \cap \sigma^{-1}(0)(\mathbb{R}))+\sum\limits_{j} \frac{\mu(E_i(\mathbb{R}) \cap E_j(\mathbb{R}))}{\alpha_j}\right)$$
where the sum runs over all $j \in J_{\mathbb{R}}\setminus\{i\}$ such that $E_i(\mathbb{R}) \cap E_j(\mathbb{R}) \neq \emptyset$, and where we set $\alpha_j = \nu_j-\frac{\nu_i}{N_i}N_j \in \mathbb{Q}^*$.
The residue of $Z_{top,0}(f;s)$ at $s_0$ is then
$$\text{Res}(Z_{top,0};s_0) = \sum\limits_{i} \mathcal{R}_{top,i}$$
where the sum now runs over the $i \in J_{\mathbb{R}}$ such that $s_0=-\frac{\nu_i}{N_i}$.
\end{notation}

If $E_i$ is an irreducible component of the strict transform, then $E_i^0 \cap \sigma^{-1}(0) (\mathbb{R}) = \emptyset$, since $\sigma^{-1}(0)$ is the union of the exceptional curves. Therefore,
$$\mathcal{R}_{top,i} = \frac{1}{N_i}\sum\limits_{j} \frac{\mu(E_i(\mathbb{R}) \cap E_j(\mathbb{R}))}{\alpha_j}$$
where the sum runs over $j \in J_{\mathbb{R}}$ such that $E_i(\mathbb{R}) \cap E_j(\mathbb{R}) \neq \emptyset$.\\
Now assume that $E_i$ is an exceptional curve. Suppose that $E_i$ intersects at $k$ real points the components $E_1,\dots,E_k$  and at $r$ complex points the components $E_{k+1},\dots,E_{k+r}$, so that $(E_i \cdot \sum\limits_{j \neq i } E_j) = k + 2r$.
One has $E_i^0 \cap \sigma^{-1}(0) (\mathbb{R}) = E_i^0(\mathbb{R}) \simeq \mathbb{P}^1(\mathbb{R}) \setminus \{ k \text{ points}\}$ so that $\beta (E_i^0\cap \sigma^{-1}(0)(\mathbb{R})) = u+1-k$ and $\mu(E_i^0\cap \sigma^{-1}(0)(\mathbb{R})) = 2-k$.
The contribution of $E_i$ to the residue of $Z_{top}(f;s)$ at $s_0$ is therefore
$$\mathcal{R}_{top,i}=\frac{1}{N_i}(2-k+\sum_{j=1}^k \frac{1}{\alpha_j}).$$

\begin{rem}
Let us consider the zeta function defined at the level of the virtual Poincaré polynomial which can be written as

$$Z_{\beta,0}(f;u^{-s}) = \sum\limits_{i \in J_{\mathbb{R}}} \frac{\beta(E^{0}_i \cap \sigma^{-1}(0)(\mathbb{R}))(u-1)}{u^{(\nu_i+sN_i)}-1} + \sum\limits_{\{i,j\} \subset J_{\mathbb{R}}} \frac{\beta(E_i(\mathbb{R}) \cap E_j(\mathbb{R}))(u-1)^2}{(u^{(\nu_i+sN_i)}-1)(u^{(\nu_j+sN_j)}-1)}.$$
It follows that the contribution of $E_i$ to the residue of $Z_{\beta,0}(f;u^{-s})$ at $s_0=-\frac{\nu_i}{N_i}$ is given by
$$\mathcal{R}_{\beta,i}(u) =  \frac{1}{N_iu^{\frac{\nu_i}{N_i}}} \left( \beta(E_i \cap \sigma^{-1}(0)(\mathbb{R})) + \sum\limits_j  \beta(E_i(\mathbb{R}) \cap E_j(\mathbb{R}))\frac{u-1}{u^{\alpha_j}-1}\right).$$
Since $u^{\frac{\nu_i}{N_i}}$ is an invertible factor depending only on $s_0=-\frac{\nu_i}{N_i}$, we will systematically omit this term in the expression of $\mathcal{R}_{\beta,i}$ in what follows, as this does not affect our study of the poles.
Using the above notation in the case where $E_i$ is an exceptional curve, one has
$$\mathcal{R}_{\beta,i}(u) = \frac{1}{N_i}(u+1-k + \sum\limits_{j=1}^k \frac{u-1}{u^{\alpha_j}-1})$$
which can be viewed as a $\mathcal{C}^\infty$ function of the real variable $u \in \mathbb{R}^*_+ \setminus \{1\}$. Furthermore, this function admits a limit as u tends to 1, which, by definition of the topological zeta function, is given by
$\underset{u \to 1}{\lim} \mathcal{R}_{\beta,i}(u) = \mathcal{R}_{top,i}$. Therefore, one has the inclusion 
$$\text{Poles}(Z_{top,0}(f;s)) \subset \text{Poles}(Z_{\beta,0}(f;u^{-s}))$$ as already mentionned in Remark \ref{rem 5}.

\end{rem}

The following proposition is a reformulation of Lemma II.2 of \cite{loes3}, adapted to the real framework.

\begin{prop}
\label{prop 3}
    With the above notation, one has 
    $$\sum_{j=1}^k \alpha_j  + 2 \sum_{j=k+1}^{k+r}\alpha_j = k + 2r -2.$$
\end{prop}

\begin{proof}
    Let us briefly review the proof of this result by Veys in \cite{veys3}. Consider Pic($X$) equipped with the bilinear intersection form. Since the divisor $D=\sum\limits_{j \in J} N_j E_j$ is principal by definition, we have $D=0$ in Pic($X$). It follows that
    $$0 = (E_i \cdot D)= \sum\limits_{j \in J} N_j (E_i \cdot E_j) = N_iE_i^2 + \sum_{j=1}^k N_j  + 2\sum_{j=k+1}^{k+r}N_j$$
    that is,
    $$\sum_{j=1}^k N_j  + 2\sum_{j=k+1}^{k+r}N_j = -N_iE_i^2.$$
    With the notation introduced, one has K$_X = \sum\limits_{j \in J}(\nu_j-1)E_j$ and the adjunction formula gives
    $$-2 = \text{deg} \text{ K}_{E_i} = E_i \cdot (K_X + E_i) = (\nu_i -1) E_i^2 + \sum\limits_{ j = 1 }^k (\nu_j -1) + 2 \sum\limits_{j=k+1}^{k+r}(\nu_j-1) + E_i^2$$
    that is 
    $$\sum\limits_{ j = 1 }^k \nu_j + 2 \sum\limits_{j=k+1}^{k+r}\nu_j = k + 2r -2 -\nu_iE_i^2.$$
    The proposition follows by combining these two equalities since $\alpha_j = \nu_j - N_j\frac{\nu_i}{N_i}.$
\end{proof}

\begin{prop}[\cite{loes3} Proposition II.3.1]
\label{prop 2}
    For all $j \in \llbracket 1, k + r \rrbracket$, we have $-1 \leq \alpha_j < 1$, equality occurring if and only if $r=0$ and $k=1$.
\end{prop}

\begin{rem}
    Let us mention that the previous result was first proven by Igusa in \cite{igusa} in the irreducible case, and later by Loeser in the general case. This proposition is where the use of the canonical embedded resolution is crucial. For example, starting from the canonical embedded resolution, if we further blow-up a point lying on an exceptional curve, we create a new $\alpha_j$ equal to 1.
\end{rem}

The following corollary then follows almost immediately from the last two results.

\begin{cor}
\label{cor 1} \begin{enumerate}
    \item For all $ j \in \llbracket k +1, k+r \rrbracket $, we have $\alpha_j \geq 0$.
    \item There is at most one $j \in \llbracket 1, k \rrbracket$ such that $\alpha_j <0$.
    \item If $k+2r \geq 3$, there is at most one $j \in \llbracket 1, k \rrbracket$ such that $\alpha_j \leq 0 $ and for all $j \in \llbracket k+1, k+r \rrbracket$ we have $\alpha_j >0$.
    \item Assume that $k+2r \geq 3$ and that there exists $\alpha_1 < 0$. Then 
    $$-\alpha_1 < \underset{2 \leq j \leq k}{\text{min}} \alpha_j $$
    \item If $k+2r = 2$, we have $k=2$ and $r=0$ and  
    $$\frac{\nu_1}{N_1} < \frac{\nu}{N} \iff \frac{\nu}{N} < \frac{\nu_2}{N_2}.$$
    \end{enumerate}
\end{cor}

\begin{proof}
    Let us prove only the fourth point, since the others can be proven in a completely similar way.
    Suppose that $-\alpha_1 \geq \underset{2 \leq j \leq k}{\text{min}} \alpha_j$. Without loss of generality we can assume that this minimum is achieved by $\alpha_2$. It then follows that 
    $$\sum\limits_{j=1}^k \alpha_j \leq \sum\limits_{j=3}^k \alpha_j \leq k-2$$
    where the latter inequality is strict if and only if $k \geq 3$. On the other hand, one knows that 
    $$2\sum_{j=k+1}^{k+r} \alpha_j \leq 2r$$ where the inequality is strict if and only if $r \geq 1$. Since $k+2r \geq 3$, at least one of these two inequalities is strict, and one finds 
    $$\sum_{j=1}^k \alpha_j + 2\sum_{j=k+1}^{k+r} \alpha_j < k+2r - 2$$
    which contradicts Proposition \ref{prop 3} and completes the proof.

\end{proof}

\begin{thm}
\label{thm 3}
     Using the notation above, the contribution $\mathcal{R}_{top,i}$ is nonzero if and only if\\
     $(E_i \cdot \sum\limits_{j \neq i} E_j) \geq 3$. In this case, one has 
    \begin{enumerate}
        \item $\mathcal{R}_{top,i} > 0$ if and only if $\alpha_j > 0$ for all $j \in \llbracket 1, k \rrbracket$, and we furthermore have $\mathcal{R}_{top,i} \geq \frac{2}{N_i}$.
        \item $\mathcal{R}_{top,i} < 0 $ if and only if there exists $j \in \llbracket 1, k \rrbracket$ such that $\alpha_j < 0$.
    \end{enumerate}
\end{thm}

\begin{proof}
    Assume that $(E_i \cdot \sum\limits_{ j \neq i } E_j) =1$. Then $E_i$ intersects another component $E_1$ at a real point, and according to Proposition \ref{prop 3}, one has $\alpha_1=-1$. Therefore,
    $$\mathcal{R}_{top,i}=\frac{1}{N_i}(1+\frac{1}{\alpha_1})=0.$$
    Assume that $(E_i \cdot \sum\limits_{j \neq i } E_j) = 2$. Then $E_i$ intersects other components $E_1, E_2$ at exactly two real points and, again according by Proposition \ref{prop 3}, one has $\alpha_1+\alpha_2=0$. It follows that
    $$\mathcal{R}_{top,i}=\frac{1}{N_i}(\frac{1}{\alpha_1}+\frac{1}{\alpha_2})=0.$$
    Now consider the case where $(E_i \cdot \sum\limits_{j \neq i} E_j) \geq 3$. If $0 < \alpha_j <1$ for all $j \in \llbracket 1, k \rrbracket$ then $\sum\limits_{j=1}^k \frac{1}{\alpha_j} - k \geq 0$ and it follows that $\mathcal{R}_{top,i} \geq \frac{2}{N_i} > 0$.\\
    Finally, assume that there exists $j \in \llbracket 1, k \rrbracket$ such that $\alpha_j < 0$. For all
    $j \in \llbracket k+1, k+r \rrbracket$ we have $\alpha_j > 0$ and we can write 
    $$N_i\mathcal{R}_{top,i} = 2-k+ \sum\limits_{j=1}^k \frac{1}{\alpha_j} = ( 2 - k-2r + \sum\limits_{j=1}^k \frac{1}{\alpha_j}  + 2\sum\limits_{j=k+1}^{k+r} \frac{1}{\alpha_j}) + (2r - 
    2\sum\limits_{j=k+1}^{k+r} \frac{1}{\alpha_j}) $$
     where $2r - 2\sum\limits_{j=k+1}^{k+r} \frac{1}{\alpha_j} = 2\sum\limits_{j=k+1}^{k+r}(1-\frac{1}{\alpha_j}) \leq 0$. On the other hand, the term $$ 2 - k-2r + \sum\limits_{j=1}^k \frac{1}{\alpha_j}  + 2\sum\limits_{j=k+1}^{k+j} \frac{1}{\alpha_j}$$ which corresponds to $\mathcal{R}_{top,i}(f_{\mathbb{C}})$, is strictly negative according to the following lemma (by setting $l = k+2r $ and $x_i = \alpha_i$), so the desired result follows.
\end{proof}

\begin{lem}[\cite{veys1} Lemma 2.9.]
\label{lem 1}
    Let $l \geq 3$ and $ -1 \leq x_1, \dots, x_l <1 $ be nonzero real numbers such that $\sum\limits_{i=1}^l x_i= l-2$. Assume that there exists a unique $i \in \llbracket 1, l \rrbracket$ such that $x_i <0$. Then 
    $$2-l+\sum\limits_{i=1}^l \frac{1}{x_i} < 0.$$
\end{lem}

\begin{rem}
    When $(E_i \cdot \sum\limits_{j \neq i } E_j) < 3$, one can check that the contribution $\mathcal{R}_{\beta,i}$ of $E_i$ to the residue of $Z_{\beta,0}(f;u^{-s}) $ at $-\frac{\nu_i}{N_i}$ is also zero. Indeed, if $(E_i \cdot \sum\limits_{j \neq i } E_j) = 1$, then $\alpha_1 = -1$, one finds that 
    $$N_i\mathcal{R}_{\beta,i}(u)=u+1-1 + \frac{u-1}{u^{\alpha_1}-1} = \frac{u^{\alpha_1  + 1}-1}{u^{\alpha_1}-1} = 0.$$
    Similarly, if $(E_i \cdot \sum\limits_{j \neq i } E_j) = 2$, then $\alpha_1 + \alpha_2 = 0$ and one finds 
    $$N_i\mathcal{R}_{\beta,i}(u) = u+1-2 + \frac{u-1}{u^{\alpha_1} -1} + \frac{u-1}{u^{\alpha_2}-1} = \frac{u^{\alpha_1+\alpha_2+1}-u^{\alpha_1+\alpha_2}-u+1}{(u^{\alpha_1}-1)(u^{\alpha_2}-1)}  = 0.$$
\end{rem}

\subsection{The real dual graph of a resolution}
\label{subsection 2.2}

\begin{defi}
\label{def 5}
    The real dual graph of the canonical embedded resolution $\sigma:(X,\sigma^{-1}(0)) \to (\mathbb{A}^2_{\mathbb{R}},0)$ is defined as the graph whose vertices are the set 
    $$\{~ i \in J_{\mathbb{R}} \mid E_{i} \text{ is an exceptional curve}~\}$$
    with an edge connecting vertices $i$ and $j$ if and only if $E_i$ and $E_j$ intersect (in which case the intersection is a single real point).
    In the real \textit{total} dual graph of the resolution, denoted $G$, each analytically irreducible component of the strict transform with a non-empty real locus is also represented by a circle, and an edge is added to connect it to the unique exceptional curve it intersects.
    To each vertex $i$, we also attach the number $\frac{\nu_i}{N_i} \in \mathbb{Q}_+$. We denote by $\mathcal{M}$ the set of minimal vertices, that is 
    $$\mathcal{M}= \{~ i \in J_{\mathbb{R}} \mid\frac{\nu_i}{N_i} = \underset{j \in J_{\mathbb{R}}}{\text{min}} \frac{\nu_j}{N_j}~\}$$
\end{defi}

\begin{ex}
\label{ex 3} 
   Let us consider the cusp as in Example \ref{ex 4}. Then the real total dual graph of the resolution is 
   \begin{figure}[H]
        \centering
        \includegraphics[scale=0.35]{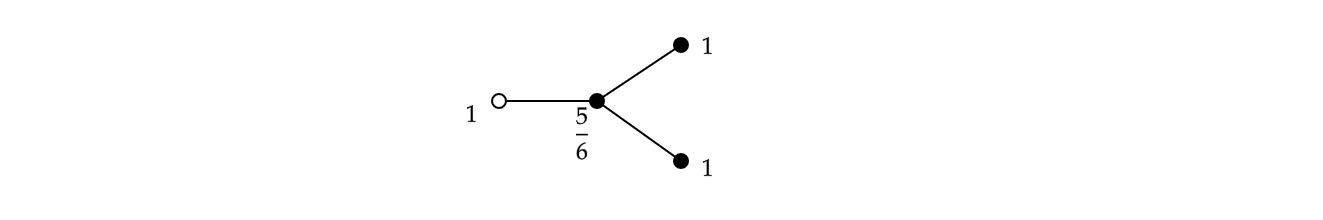}
    \end{figure}
    
\end{ex}

\begin{rem}
\label{rem 8}
    It is well known that the dual graph of the resolution is always a tree, i.e., a connected and acyclic graph. To see this, one can follow the evolution of the graph during the resolution process and check that, by blowing-up a point lying on an exceptional curve, the graph will be modified into a graph that is still connected. To obtain the real total dual graph of the resolution from the dual graph of the resolution, we simply add the vertices corresponding to the non-empty real loci of the analytically irreducible components of the strict transform and add edges representing the real intersection points between these components and the exceptional curves. In particular, this construction shows that that the real total dual graph is also a tree.
\end{rem}

We will picture an exceptional curve that intersects at least one other component at a real point as
     \begin{figure}[H]
        \centering
        \includegraphics[scale=0.35]{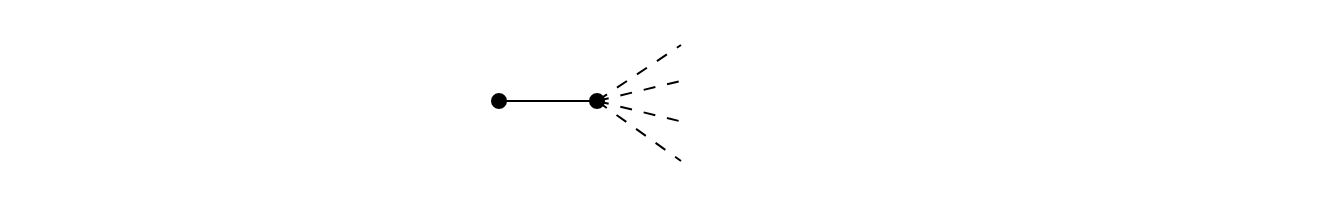}
    \end{figure}
    
The proposition below follows immediately from Corollary \ref{cor 1}.
\begin{prop}
\label{prop 4}
Assume that in $G$ one has
    \begin{figure}[H]
        \centering
        \includegraphics[scale=0.35]{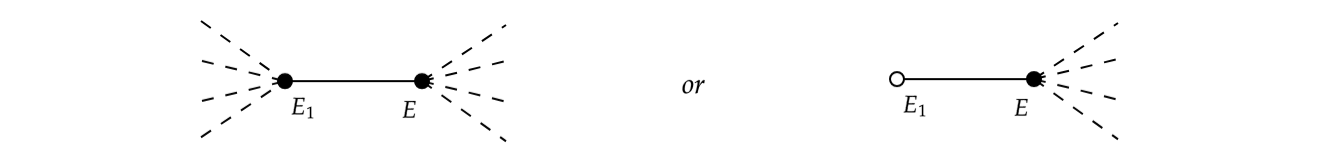}
    \end{figure}
with $\frac{\nu_1}{N_1} <  \frac{\nu}{N}$. Then $\frac{\nu}{N} < \frac{\nu_i}{N_i}$ for any other component $E_i$ that intersects $E$.
\end{prop}

By induction one obtain the following corollary.

\begin{cor}
\label{cor 2} \begin{enumerate}
    \item Consider a path in $G$ that starts at a vertex in $\mathcal{M}$ and immediately leaves $\mathcal{M}$. Then the numbers $\frac{\nu_i}{N_i}$ strictly increase along the path.
    \item The minimal part $\mathcal{M}$ forms a connected subgraph of $G$.
    \end{enumerate}
\end{cor}

\begin{thm}
\label{thm 4}
    Let $s_0 \in \mathbb{Q}$. Then $s_0$ is a pole of $Z_{top,0}(f;s)$ if and only if $s_0=-\frac{1}{N_i}$ for some irreducible component of the strict transform $E_i$ such that $E_i(\mathbb{R}) \neq \emptyset$ or $s_0 = -\frac{\nu_i}{N_i}$ for some exceptional curve $E_i$ satisfying $(E_i \cdot \sum\limits_{j \neq i} E_j) \geq 3$. 
\end{thm}

Let us make a few remarks before we proceed to the proof of the theorem.

\begin{rem}
\label{rem 7}\begin{enumerate}
    \item Note that to determine whether an exceptional curve $E_i$ contributes to the residue of $Z_{top,0}(f;s)$ at the candidate pole $s_0=-\frac{\nu_i}{N_i}$, one must count both the $\textit{real and complex}$ intersection points of $E_i$ with the other components, even though $Z_{top,0}(f;s)$ is an invariant of the real locus of $f$. 
   \item The theorem above follows immediately from Theorem \ref{thm 3} in the case when there is only one contribution for $s_0$ induced by an exceptional curve $E_i$. In the general case, we will use the real total dual graph of the resolution to show that different contributions to the residue of $Z_{top,0}(f;s)$ at $s_0$ have the same sign so they cannot cancel each other out.
   \end{enumerate}
\end{rem}

\begin{proof}[ Proof of Theorem \ref{thm 4}]
    The fact that the condition is necessary already follows from Theorem \ref{thm 3}. Conversely, assume that there exists $i \in J_{\mathbb{R}}$ such that $s_0=-\frac{\nu_i}{N_i}$ and assume that the component $E_i$ intersects other components $E_1,\dots,E_k$ at $k$ real points. As before, let us denote $\alpha_j = \nu_j-\frac{\nu_i}{N_i}N_j$ for $j \in \llbracket 1, k \rrbracket$ and distinguish two cases.

    The first case is when $-s_0 = \underset{j \in J_{\mathbb{R}}}{\text{min}}\frac{\nu_j}{N_j}$.
    Assume that $E_i$ is an exceptional curve that satisfies $(E_i \cdot \sum\limits_{j \neq i} E_j) \geq 3$. If one of the $\alpha_j$ is zero, we have seen that $s_0$ is a pole of order 2. Otherwise, all $\alpha_j$ are strictly positive and by Theorem \ref{thm 3}, the contribution of $E_i$ to the residue of $Z_{top,0}(f;s)$ at $s_0$ is
    $$\mathcal{R}_{top,i} = \frac{1}{N_i}(2-k+\sum\limits_{j=1}^k \frac{1}{\alpha_j}) \geq \frac{2}{N_i} > 0 .$$

    Now assume that $E_i$ is a component of the strict transform. Again, if one of the $\alpha_j$ is zero, then $s_0$ is a pole of order 2. Otherwise, all $\alpha_j$ are strictly positive and the contribution of $E_i$ to the residue of $Z_{top,0}(f;s)$ at $s_0$ is 
    $$\mathcal{R}_{top,i} = \frac{1}{N_i}\sum\limits_{j=1}^k \frac{1}{\alpha_j} > 0 .$$
    All nonzero contributions to the residue of $Z_{top,0}(f;s)$ at $s_0$ are therefore strictly positive, thus $s_0$ is a pole of $Z_{top}(f;s)$.

    The second case is when $-s_0 > \underset{j \in J_{\mathbb{R}}}{\text{min}}\frac{\nu_j}{N_j}$. Assume that $E_i$ is an exceptional curve that satisfies $(E_i \cdot \sum\limits_{j \neq i} E_j) \geq 3$.
    By assumption, the vertex $i \in G$ is not in $\mathcal{M}$, and by connectedness, there exists an elementary path starting from a vertex of $\mathcal{M}$ and ending at $i$. By Corollary \ref{cor 2}, there exists $j_0 \in \llbracket 1, k \rrbracket$ such that $\alpha_{j_0} = \nu_{j_0}-\frac{\nu_i}{N_i}N_{j_0} < 0$. By Theorem \ref{thm 3}, the contribution of $E_i$ to the residue of $Z_{top,0}(f;s)$ at $s_0$ is:
    $$\mathcal{R}_{top,i} = \frac{1}{N_i}(2-k+\sum\limits_{j=1}^k \frac{1}{\alpha_j}) < 0 .$$
    Finally, if $E_i$ is a component of the strict transform, we know from the previous corollary that all $\alpha_j$ are strictly negative. The contribution of $E_i$ to the residue of $Z_{top,0}(f;s)$ at $s_0$ is
    $$\mathcal{R}_{top,i} = \frac{1}{N_i}\sum\limits_{j=1}^k \frac{1}{\alpha_j} < 0$$ 
    All nonzero contributions to the residue of $Z_{top,0}(f;s)$ at $s_0$ are therefore strictly negative, thus, $s_0$ is a pole of $Z_{top,0}(f;s)$.
\end{proof}

\begin{cor}
    \label{cor 4}
    We have the following equalities
    $$\text{Poles}(Z_{top,0}(f;s)) = \text{Poles}(Z_{\beta,0}(f;u^{-s})) = \text{Poles}(Z_{mot,0}(f;\mathbb{L}^{-s})).$$
    
\end{cor}

\begin{proof} We already have the inclusions 
    $$\text{Poles}(Z_{top,0}(f;s)) \subset \text{Poles}(Z_{\beta,0}(f;u^{-s})) \subset \text{Poles}(Z_{mot,0}(f;\mathbb{L}^{-s}))$$
    so we only need to prove that $\text{Poles}(Z_{mot,0}(f;\mathbb{L}^{-s})) \subset \text{Poles}(Z_{top,0}(f;s))$. In other words, we take a candidate pole $s_0 \in \mathbb{Q}$ which is not a pole of $Z_{top,0}(f;s)$ and we must check that $s_0$ is not a pole of $Z_{mot,0}(f;s)$. By Theorem \ref{thm 4}, if $s_0$ is not a pole of $Z_{top,0}(f;s)$, then $s_0=-\frac{\nu_i}{N_i}$ for an exceptional curve $E_i$ that satisfies $(E_i \cdot \sum\limits_{j \neq i } E_j) < 3$.\\
    Assume that $(E_i \cdot \sum\limits_{j \neq i } E_j) = 1$. Then $E_i$ intersects another component $E_1$ at a single real point and one has $\alpha_1 = \nu_1 - \frac{\nu_i}{N_i}N_1 = -1$ by Proposition \ref{prop 3}. The contribution of $E_i$ to the residue of $Z_{mot,0}(f;T)$ at $s_0$ comes from the term
    $$(\mathbb{L}-1)\mathbb{L}\frac{\mathbb{L}^{-\nu_i}T^{N_i}}{1 - \mathbb{L}^{-\nu_i} T^{N_i}} + (\mathbb{L}-1)^2 \frac{\mathbb{L}^{-\nu_i}T^{N_i}}{1 - \mathbb{L}^{-\nu_i}T^{N_i}} \frac{\mathbb{L}^{-\nu_1} T^{N_1}}{1 - \mathbb{L}^{-\nu_1}T^{N_1}} $$ which equals 
    $$\frac{\mathbb{L}(\mathbb{L}-1)\mathbb{L}^{-\nu_i}T^{N_i}(1- \mathbb{L}^{-(\nu_1 + 1) } T^{N_1}) }{(1 - \mathbb{L}^{-\nu_i}T^{N_i}) ( 1 - \mathbb{L}^{-\nu_1}T^{N_1})}.$$
    The numerator is a multiple of the term $1 - \mathbb{L}^{-\frac{\nu_i}{N_i}N_1} T^{N_1} = 1 - (\mathbb{L}^{-\nu_i}T^{N_i})^{\frac{N_1}{N_i}} $ which is a fortiori a multiple of $1 - \mathbb{L}^{-\nu_i} T^{N_i}$ in K$_0(\mathbb{R}\text{Var})$, therefore $E_i$ does not contribute to the residue of $Z_{mot,0}(f;\mathbb{L}^{-s})$ at $s_0$.
    
    Now assume that $(E_i \cdot \sum\limits_{j \neq i } E_j) = 2$. Then $E_i$ intersects others components $E_1, E_2$ at two real points and one has $\alpha_1  + \alpha_2 = 0 $ by Proposition \ref{prop 3}, which gives $\nu_1 + \nu_2 = \frac{\nu_i}{N_i}( N_1 + N_2 )$. The contribution of $E_i$ to the residue of $Z_{mot,0}(f;T)$ at $s_0$ comes from the term

    $$(\mathbb{L}-1)^2\frac{\mathbb{L}^{-\nu_i}T^{N_i}}{1 - \mathbb{L}^{-\nu_i}T^{N_i}} + (\mathbb{L}-1)^2 \frac{\mathbb{L}^{-\nu_i}T^{N_i}}{1 - \mathbb{L}^{-\nu_i}T^{N_i}} \frac{\mathbb{L}^{-\nu_1}T^{N_1}}{1 - \mathbb{L}^{-\nu_1}T^{N_1}} + (\mathbb{L}-1)^2 \frac{\mathbb{L}^{-\nu_i}T^{N_i}}{1 - \mathbb{L}^{-\nu_i}T^{N_i}} \frac{\mathbb{L}^{-\nu_2}T^{N_2}}{1 - \mathbb{L}^{-\nu_2}T^{N_2}}$$
    which equals
    $$\frac{(\mathbb{L}-1)^2 \mathbb{L}^{-\nu_i}T^{N_i}(1-\mathbb{L}^{-(\nu_1+\nu_2)}T^{N_1 + N_2} )  }{(1 - \mathbb{L}^{-\nu_i}T^{N_i}) ( 1 - \mathbb{L}^{-\nu_1}T^{N_1} ) ( 1 - \mathbb{L}^{-\nu_2}T^{N_2} )  }$$
     Thus, the numerator is a multiple of the term $1 - \mathbb{L}^{-\frac{\nu_i}{N_i}(N_1+ N_2)}T^{N_1 + N_2 } = 1 - (\mathbb{L}^{-\nu_i}T^{N_i})^{\frac{N_1 + N_2}{N_i}} $ which is a fortiori a multiple of $1 - \mathbb{L}^{-\nu_i}T^{N_i}$ in K$_0(\mathbb{R}\text{Var})$. Therefore, $E_i$ does not contribute to the residue of $Z_{mot,0}(f;\mathbb{L}^{-s})$ at $s_0$.

\end{proof}

Comparing Theorem \ref{thm 4} with Theorem 4.3 of \cite{veys1} mentioned in the introduction, we obtain the following corollary.

\begin{cor}
    One has the following equality
    $$\text{Poles}(Z_{top,0}(f;s)) = \text{Poles}(Z_{top,0}(f_{\mathbb{C}};s)) \cap  \{~ -\frac{\nu_i}{N_i} \mid  i \in J_{\mathbb{R}}~\}$$
    and the same equality holds for the motivic zeta functions.
\end{cor}
    
\begin{ex}
\label{ex 2}
Take $f = (x^2+y^6)^2(x^2-y^3)^3$. By performing four successive blowings-up, we obtain an embedded resolution $\sigma :(X,\sigma^{-1}(0)) \to (\mathbb{A}^2_{\mathbb{R}},0)$ of $f$. Scheme-theoretically, the strict transform is the union of two smooth irreducible curves $E_5$ and $E_6$, where $E_5$ has only real points and $E_6$ has no real points. The resolution graph showing only real points is of the following form
    \begin{figure}[H]
        \centering
        \includegraphics[scale=0.35]{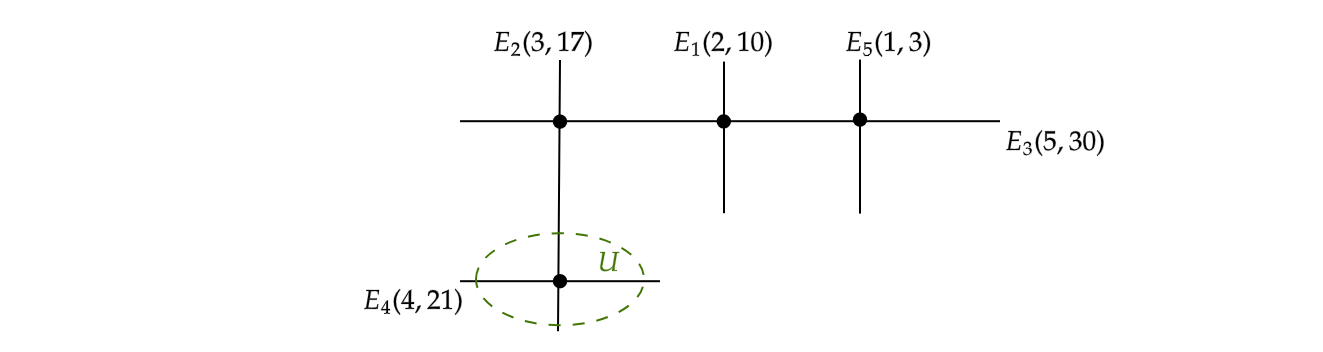}
    \end{figure}
so the candidate poles are $-\frac{4}{21},-\frac{3}{17}, -\frac{2}{10},-\frac{1}{3},-\frac{5}{30}$.
The $E_6$ component is contained in a Zariski open neighborhood $U \simeq \mathbb{A}^2_{\mathbb{R}}$ of $E_2 \cap E_4$ on which $f(\sigma(x_1,y_1)) = y_1^{21}x_1^{17} (x_1^2+1)^2$, so that 
$(E_4 \cdot \sum\limits_{j \neq 4}E_j) = (E_4 \cdot E_2) + (E_4 \cdot E_6)= 1+2 = 3$. By Theorem \ref{thm 4} the poles of $Z_{top,0}(f;s)$ are $-\frac{4}{21}, -\frac{1}{3}$ and $-\frac{5}{30}$. One can also compute $Z_{top,0}(f;s)$ using the resolution graph, and after simplification one has 
$$Z_{top}(f;s)=\frac{20+141s+216s^2}{(5+30s)(4+21s)(1+3s)}$$
which is consistent with our study of the poles. One can also check that 
Res$(Z_{top,0}; - \frac{5}{30}) > 0$, Res$(Z_{top,0}; -\frac{1}{3}) < 0$ and Res$(Z_{top,0};-\frac{4}{21}) <0$, which is also consistent with Theorem \ref{thm 4}, given that the real total dual graph of the resolution is 
    \begin{figure}[H]
        \centering
        \includegraphics[scale=0.35]{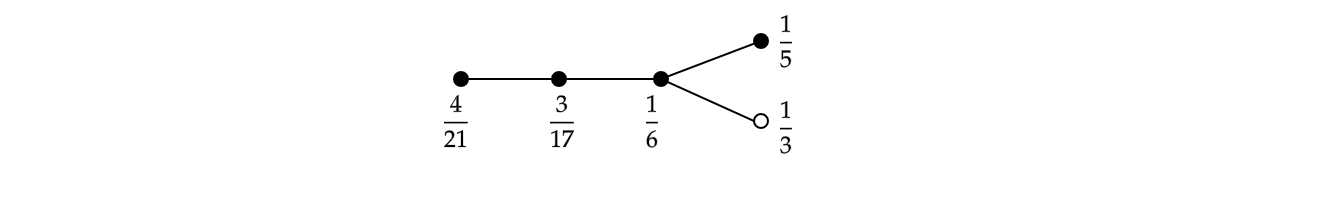}
    \end{figure}

In this example, the poles of $Z_{top,0}(f_{\mathbb{C}};s)$ are $-\frac{5}{30},-\frac{4}{21},-\frac{1}{3}$ and $-\frac{1}{2}$.

\end{ex}

\section{Poles of zeta functions with signs for curves}
\label{section 3}
In this section, we study the poles of the zeta functions with signs using the same approach as in Section \ref{section 2}. In particular, we show that every pole of the topological zeta functions with signs is also a pole of the naive topological zeta function. On the other hand, certain cancellations may occur at the level of the topological zeta functions, which leads us to study more precisely the residue at the level of the virtual Poincaré polynomial. We obtain a description of the poles in the case where there is at most one nonzero contribution for a given candidate pole.

Let us temporarily consider $f \in \mathbb{R}[x_1,\dots,x_d]$ and let $\sigma : (X, \sigma^{-1}(0)) \to (\mathbb{A}^d_{\mathbb{R}},0)$ be an embedded resolution of $f$. By definition, one has 
    $$Z_{top,0}^{\pm}(f;s)= \sum\limits_{ \emptyset \neq I \subset J} \mu(\widetilde{E}_I^{0,\pm} \cap \sigma^{-1}(0)(\mathbb{R})) \underset{i \in I}{\prod}\frac{1}{\nu_i+sN_i}$$
    and if $I$ is not included in $J_{\mathbb{R}}$, we have seen that $E_I^0 \cap \sigma^{-1}(0) ( \mathbb{R}) = \emptyset$ which implies that $\widetilde{E}_I^{0,\pm} \cap \sigma^{-1}(0)(\mathbb{R}) = \emptyset$. Therefore,
    $$Z_{top,0}^{\pm}(f;s)= \sum\limits_{ \emptyset \neq I \subset J_{\mathbb{R}}} \mu(\widetilde{E}_I^{0,\pm} \cap \sigma^{-1}(0)(\mathbb{R})) \underset{i \in I}{\prod}\frac{1}{\nu_i+sN_i}.$$

\begin{rem} 
    \label{rem 10}
    Let us denote the positive and negative parts of $f$ by $P(f)$ and $N(f)$ respectively, i.e.
    $$P(f) = \{ ~x \in X(\mathbb{R}) \mid (f \circ \sigma)(x) > 0 ~\} \text{ and } N(f) = \{~ x \in X(\mathbb{R}) \mid (f \circ \sigma)(x) < 0 ~\}.$$
    We then define the subsets $J_{\mathbb{R}}^{\pm} \subset J_{\mathbb{R}}$ by 
    $$J_{\mathbb{R}}^{+} = \{~ j \in J \mid E_{j}\ (\mathbb{R}) \cap \overline{P(f)} \neq \emptyset ~\} \text{ and } J_{\mathbb{R}}^{-} = \{~ j \in J \mid E_{j}\ (\mathbb{R}) \cap \overline{N(f)} \neq \emptyset ~\}.$$
    Let $I \subset J_{\mathbb{R}}$ and suppose that there exists $i \in I$ such that $i \notin J_{\mathbb{R}}^ {+}$, that is, $f \circ \sigma$ is negative in a neighborhood of $E_i(\mathbb{R})$. A fortiori, $f \circ \sigma$ is negative in the neighborhood of $E_I^0(\mathbb{R})$. We will see that $\widetilde{E}_I^{0,+}(\mathbb{R})$ is empty, and we only need to check this locally. Let $U$ be a Zariski open set such that $f \circ \sigma = u \prod\limits_{i \in I} y_i^{N_i}$ on $E_I^0 \cap U$ and where $u$ is a unit, i.e, $u$ does not vanish on $E_I^0 \cap U$. The fact that $f\circ \sigma$ is negative in a neighborhood of $E_I^0(\mathbb{R})$ implies that all $N_i$ are even and that $u < 0$ on $E_I^0(\mathbb{R})$. Denote $m = \text{gcd}(N_i)$, which is even. Then one has
    $$\widetilde{E}_I^{0,+}(\mathbb{R})\cap U \simeq R_U^{+} = \{ ~(x,t) \in (E_I^0(\mathbb{R}) \cap U)\times\mathbb{R} \mid t^mu(x) =  1 ~\} = \emptyset$$
    so we can write 
    $$Z_{top,0}^{+}(f;s)= \sum\limits_{ \emptyset \neq I \subset J_{\mathbb{R}}^+} \mu(\widetilde{E}_I^{0,+} \cap \sigma^{-1}(0)(\mathbb{R})) \underset{i \in I}{\prod}\frac{1}{\nu_i+sN_i}.$$
    Similarly, one has
    $$Z_{top,0}^{-}(f;s)= \sum\limits_{ \emptyset \neq I \subset J_{\mathbb{R}}^-} \mu(\widetilde{E}_I^{0,-} \cap \sigma^{-1}(0)(\mathbb{R})) \underset{i \in I}{\prod}\frac{1}{\nu_i+sN_i}.$$

    It follows that the set of candidate poles of $Z_{top,0}^{\pm}(f;s)$, and thus the set of candidate poles of $Z_{mot,0}^{\pm}(f;\mathbb{L}^{-s})$), is  
    $$\{~ -\frac{\nu_i}{N_i} \mid  i \in J_{\mathbb{R}}^{\pm}~\}.$$
\end{rem}

\begin{ex}
  Let us consider $f= y^2-x^3$ as in Example \ref{ex 4}. The resolution graph on which we illustrate the positive and negative parts of $f$ as well as the coverings $\widetilde{E}_i^{0,+}(\mathbb{R})$ is as follows.

    \begin{figure}[H]
        \centering
        \includegraphics[scale=0.35]{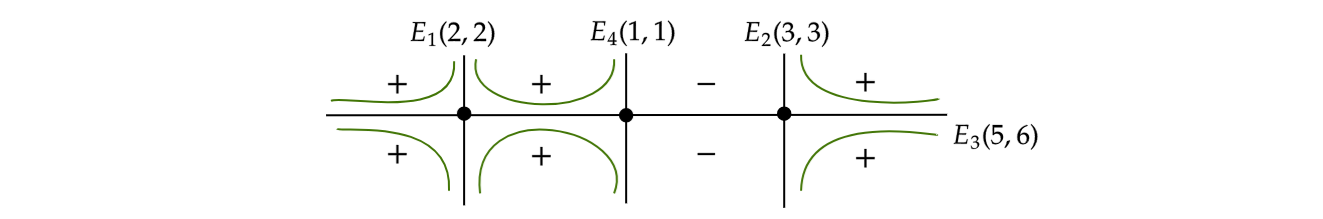}
    \end{figure}
On the intersections $E_{i,j}^0=E_i \cap E_j$, the term $\widetilde{E}_{i,j}^{0,\pm}(\mathbb{R})$ consists of $0,1$ or $2$ points depending on the parity of $\text{gcd}(N_i, N_j)$ and the sign of $f$ in a neighborhood of $E_i \cap E_j$. To compute $\widetilde{E}_i^{0,\pm}(\mathbb{R})$, one can use the local charts provided by the sequence of blowings-up.\\
For example, $E_1^0(\mathbb{R}) = \mathbb{P}^1(\mathbb{R)} \setminus \{~1 \text{ point } \}$ is contained in a single affine chart $U \simeq\mathbb{A}^2_{\mathbb{R}}$ where $f(\sigma(u_1,v_1)) = v_1^2(1-u_1^3v_1)$ with $E_1 = \{v_1 = 0\}$ and $w(u_1,v_1) = 1-u_1^3v_1$ is a unit, i.e., $w$ does not vanish on $E_1$. By definition, one has
$$\widetilde{E}_1^{0,+}(\mathbb{R}) = \{~ (u_1,t) \in \mathbb{R}^2 \mid t^2w(u_1,0) = 1~\} = \{~ (u_1,t) \in \mathbb{R}^2 \mid t^2 = 1~\}$$ 
which is isomorphic to two disjoint copies of $\mathbb{R}$. Hence $\beta(\widetilde{E}_1^{0,+}(\mathbb{R})) = 2 u$ and $\mu(\widetilde{E}_1^{0,+}(\mathbb{R})) = 2$. The other terms  $\widetilde{E}_i^{0,+}(\mathbb{R})$ can be computed similarly, yielding 
$$Z_{top,0}^{+}(f;s) = \frac{2}{2+2s} + \frac{1} {3+3s} - \frac{2}{5+6s} + \frac{2}{(2+2s)(5+6s)} + \frac{1}{(3+3s)(5+6s)} + \frac{1}{(1+s)(5+6s)}$$
which simplifies to
$$Z_{top,0}^{+}(f;s) = \frac{6s + 7}{(s+1)(5+6s)}.$$
It follows that the poles of $Z_{top,0}^{+}(f;s)$ are $-1$ and $-\frac{5}{6}$.
Similarly, one computes 
$$Z_{top,0}^{-}(f;s) = \frac{1}{3+3s} + \frac{1}{(3+3s)(5+6s)} + \frac{1}{(1+s)(5+6s)} = \frac{2s +3}{(s+1)(5+6s)}$$
so the poles of $Z_{top,0}^{-}(f;s)$ are also $-1$ and $-\frac{5}{6}$.

\end{ex}

\begin{rem}\begin{enumerate}
    \item It may happen that the set of poles of $Z_{top,0}^{\pm}(f;s)$ is strictly included in the set of poles of $Z_{top,0}(f;s)$. For instance, if $f \leq 0$, then $Z_{top,0}^+(f;s) = 0$, and therefore $\text{Poles}(Z_{top,0}^+(f;s)) = \emptyset$. It can also occur that $Z_{top,0}(f;s)$ and $Z_{top,0}^{\pm}(f;s)$ share a pole, but that the order of this pole differs between the two functions.
    \item Contrary to what intuition might suggest and what we observed in the previous example, it is generally not true that $Z_{top,0}^{+}(f;s) + Z_{top,0}^-(f;s) = 2Z_{top,0}(f;s)$. For example, consider $f = x^2+y^6$. After performing the resolution and computing the coverings $\widetilde{E}_i^{0,\pm}(\mathbb{R})$, one obtains 
$$Z_{top,0}^-(f;s)=0,~ Z_{top,0}^+(f;s) = \frac{4}{3+4s}   \text{ while } Z_{top,0}(f;s) = \frac{3}{3+4s}.$$
\end{enumerate}
\end{rem}

\subsection{Study of a contribution : the virtual Poincaré polynomial of a real curve of hyperelliptic type}
\label{subsection 3.1}

 \text{ }\\
From now on, we fix $f \in \mathbb{R}[x,y]$ and $\sigma : (X,\sigma^{-1}(0)) \to (\mathbb{A}^2_{\mathbb{R}},0)$ be the \textit{canonical embedded resolution} of $f$. One has 
$$Z_{top,0}^{\pm}(f;s) = \sum\limits_{ \emptyset \neq I \subset J_{\mathbb{R}}^{\pm}} \mu(\widetilde{E}^{0,\pm}_I \cap \sigma^{-1}(0)(\mathbb{R})) \underset{i \in I}{\prod}\frac{1}{\nu_i+sN_i}.$$
 but $E_I^0 = \emptyset$ as soon as $|I| > 2$ since the $E_i$ are simultaneously normal crossings. Therefore,
$$Z_{top,0}^{\pm}(f;s)= \sum\limits_{i \in J_{\mathbb{R}}^{\pm}} \frac{\mu(\widetilde{E}^{0,\pm}_i \cap \sigma^{-1}(0) (\mathbb{R}))}{\nu_i+sN_i} +  \sum\limits_{\{i,j\} \subset J_{\mathbb{R}}^{\pm}} \frac{\mu(\widetilde{E}^{0,\pm}_{i,j}(\mathbb{R}))}{(\nu_i+sN_i)(\nu_j+sN_j)}.$$
To clarify the ideas and slightly simplify the notation, we will focus on the positive topological zeta function $Z_{top,0}^+(f;s)$. Since $Z_{top,0}^-(f;s) = Z_{top,0}^+(-f;s)$, all the following results have analogous statements for the function $Z_{top,0}^-(f;s)$.\\
From the expression above, one can already see that any pole of $Z_{top,0}^{+}(f;s)$ has order at most 2.

\begin{prop}
\label{prop 5} 
    Let $s_0 \in \mathbb{Q}$. Then $s_0$ is a pole of order $2$ of $Z_{top,0}^{+}(f;s)$ if and only if there exist distinct $i,j \in J_{\mathbb{R}}^+$ such that $E_i(\mathbb{R}) \cap E_j(\mathbb{R}) \cap \overline{P(f)} \neq \emptyset$ and such that $s_0=-\frac{\nu_i}{N_i}=-\frac{\nu_j}{N_j}$.
\end{prop}

\begin{proof}
    The proof is the same as that of Proposition \ref{prop 1}, noting that $\widetilde{E}_{i,j}^{0,+}(\mathbb{R})$ is non-empty if and only if $E_i(\mathbb{R}) \cap E_j(\mathbb{R}) \cap \overline{P(f)} \neq \emptyset$.
\end{proof}

\begin{rem}
Let us return to the two particular cases discussed in Remark \ref{rem 11}. Assume that $f$ is analytically equivalent to $ \lambda x^Ny^M$ with $\lambda \in \mathbb{R}^*$ and $N, M \in \mathbb{N}$. If $N$ or $M$ is odd, then
$$Z_{top,0}^+(f;s)=Z_{top,0}(f;s)=\frac{1}{(1+sN)(1+sM)}.$$
If both $N$ and $M$ are even, then 
$$Z_{top,0}^+(f;s) = 0 \text{ if } \lambda <0 \text{, and } Z_{top,0}^+(f;s)=\frac{2}{1+sN} \text{ if } \lambda>0.$$
Now, assume that there exists an exceptional curve $E_i$ such that $(E_i \cdot \sum\limits_{j \neq i } E_j) = 2$ and such that $E_i$ does not intersect any component at a real point. As we have seen, on an affine chart $U \simeq \mathbb{A}^2_{\mathbb{R}}$ of $\text{Bl}_0 \mathbb{A}^2_{\mathbb{R}}$, one can write $f(\sigma(x,y))= y^{2M}u(x,y)$, where $u(x,0) = (ax^2+bx+c)^M$ is such that $b^2-4ac < 0$. If $a < 0$ (that is, if $f \leq 0$), then $\widetilde{E}_i^{0,+}(\mathbb{R}) = \emptyset$, so
$$Z_{top,0}^+(f;s) = 0$$
If $a > 0$ (that is, if $f \geq 0$), then 
$$(\widetilde{E}_i^{0,+} \cap U )(\mathbb{R}) = \{~ (x,t) \in \mathbb{R}^2 ~|~ t^{2M}(ax^2+bx+c)^M =1 ~\}.$$
The change of variables $(u,v)=(xt,t)$ gives an isomorphism 
$$(\widetilde{E}_i^{0,+} \cap U )(\mathbb{R}) \simeq \{~ (u,v) \in \mathbb{R}^* \times \mathbb{R} ~|~ (au^2 +buv +cv^2)^M =1~\} = \{~ (u,v) \in \mathbb{R}^* \times \mathbb{R} ~|~ au^2 +buv +cv^2 =1~\}$$
because $au^2+buv+cv^2$ is positive on $\mathbb{R}^2$. Now, the curve defined by $au^2+buv+cv^2$ is an ellipse. In particular, it is a smooth compact curve homeomorphic to $\mathbb{S}^1$ which intersects the $v$ axis at two distinct points. Therefore
$$\beta((\widetilde{E}_i^{0,+} \cap U )(\mathbb{R})) = u+1-2=u-1.$$
On the other affine chart $V \simeq \mathbb{A}^2_{\mathbb{R}}$ of $\text{Bl}_0 \mathbb{A}^2_{\mathbb{R}}$, one also has $\beta((\widetilde{E}_i^{0,+} \cap V )(\mathbb{R})) =u-1$, and by additivity it follows that
$$\beta(\widetilde{E}_i^{0,+}(\mathbb{R})) = \beta((\widetilde{E}_i^{0,+} \cap U )(\mathbb{R})) + \beta((\widetilde{E}_i^{0,+} \cap V )(\mathbb{R})) - \beta((\widetilde{E}_i^{0,+} \cap U \cap V )(\mathbb{R}))$$
which gives 
$$\beta(\widetilde{E}_i^{0,+}(\mathbb{R})) = u-1 + u-1 - (u-3) = u+1.$$
Note that $\widetilde{E}_i^{0,+}(\mathbb{R})$ is a smooth compact curve homeomorphic to $\mathbb{S}^1$ and that the map $\widetilde{E}_i^{0,+} (\mathbb{R}) \to E_i^0(\mathbb{R})$ is a degree $2$ topological covering  of $\mathbb{S}^1$ by itself. One obtains
$$Z_{top,0}(f;s) = \frac{2}{2+2Ms} = \frac{1}{1+sM}.$$
\end{rem}

Throughout the rest of this section, we will assume that $f$ is not as in the above remark, which means that for all $i\in J_{\mathbb{R}}$, we have $(E_i \cdot \sum\limits_{j \neq i } E_j) \geq 1$ and that if $(E_i \cdot \sum\limits_{j \neq i} E_j) =2$, then $E_i$ intersects other components at two real points.

\begin{notation}
\label{nota 3}
As in \ref{nota 2}, let us fix $i \in J_{\mathbb{R}}^+$ and study the contribution of a component $E_i$ to the residue of $Z_{top,0}^+(f;s)$ at the candidate pole $s_0=-\frac{\nu_i}{N_i}$. We may assume that for all $j\in J_{\mathbb{R}}^+\setminus \{i\}$ such that $E_i(\mathbb{R}) \cap E_j(\mathbb{R}) \cap \overline{P(f) }\neq \emptyset$ one has $\frac{\nu_i}{N_i} \neq \frac{\nu_j}{N_j}$, otherwise we saw above that $s_0$ is a pole of order 2. We then truncate the function $Z_{top,0}^+(f;s)$ by keeping only the terms 
$$\frac{\mu(\widetilde{E}_i^{0,+} \cap \sigma^{-1}(0)(\mathbb{R}))}{\nu_i +sN_i} + \sum\limits_{j \neq i } \frac{\mu(\widetilde{E}^{0,+}_{i,j}(\mathbb{R}))}{(\nu_i+sN_i) (\nu_j+sN_j)}.$$
We denote by $\mathcal{R}_{top,i}^+$ the residue of this expression at $s_0$, that is,
$$\mathcal{R}_{top,i}^+ = \frac{1}{N_i}\left(\mu(\widetilde{E}_i^{0,+} \cap \sigma^{-1}(0)(\mathbb{R}))+\sum\limits_{j} \frac{\mu(\widetilde{E}^{0,+}_{i,j}(\mathbb{R}))}{\alpha_j}\right).$$
The residue of $Z_{top,0}^+(f;s)$ at $s_0$ is then given by
$$\text{Res}(Z_{top,0}^+;s_0) = \sum\limits_{i} \mathcal{R}_{top,i}^+$$
where the sum runs over all $i \in J_{\mathbb{R}}^+$ such that $s_0=-\frac{\nu_i}{N_i}$.
\end{notation}

If $E_i$ is an irreducible component of the strict transform, then $\widetilde{E}_i^{0,+} \cap \sigma^{-1}(0) (\mathbb{R}) = \emptyset$, since $\sigma^{-1}(0)$ is the union of the exceptional curves. Therefore,
$$\mathcal{R}_{top,i}^+ = \frac{1}{N_i}\sum\limits_{j} \frac{\mu(\widetilde{E}^{0,+}_{i,j}(\mathbb{R}))}{\alpha_j}.$$
Now assume that $E_i$ is an exceptional curve that intersects at $k$ real points other components $E_1,\dots,E_k$  and at $r$ complex points other components $E_{k+1},\dots,E_{k+r}$ so that $(E_i \cdot \sum\limits_{j\ne i } E_j) =k+2r.$ Then
$$\mathcal{R}_{top,i}^+ = \frac{1}{N_i} \left( \mu(\widetilde{E}_i^{0,+}(\mathbb{R})) + \sum\limits_{j =1}^k \frac{\mu(\widetilde{E}_{i,j}^{0,+}(\mathbb{R}))}{\alpha_j} \right) $$
where one must be careful that certain terms $\mu(\widetilde{E}_{i,j}^{0,+}(\mathbb{R}))$ may vanish.

For the zeta function defined at the level of the virtual Poincaré polynomial, the contribution of an exceptional curve $E_i$ to the residue of $Z_{\beta,0}^+(f;u^{-s})$ at $-\frac{\nu_i}{N_i}$ is given (up to an invertible depending only on $-\frac{\nu_i}{N_i}$) by 
$$\mathcal{R}_{\beta,i}^+(u)=  \frac{1}{N_i} \left( \beta(\widetilde{E}_i^{0,+}(\mathbb{R})) + \sum\limits_{j=1}^k  \beta(\widetilde{E}_{i,j}^{0,+}(\mathbb{R}))\frac{u-1}{u^{\alpha_j}-1}\right),$$
and it satisfies $\underset{u \to 1}{\lim}\mathcal{R}_{\beta,i}^+(u) = \mathcal{R}_{top,i}^+.$

\begin{rem}
    \label{rem 12} 
    We already know that $\beta(E_i^0(\mathbb{R})) = u+1-k$ so that $\mu(E_i^0(\mathbb{R})) = 2-k$. However, the computation of $\beta(\widetilde{E}_i^{0,+}(\mathbb{R}))$ is, in general, more intricate for the following reasons.
    Let $U\simeq \mathbb{R}^2$ be a Zariski open set on which $f(\sigma(x,y)) = u(x,y)y^{N_i}$. Then
    $$E_i(\mathbb{R}) \cap U = \{~ (x,y) \in \mathbb{R}^2 ~ | ~ y = 0~\} \text { and } E_i^0(\mathbb{R}) \cap U = \{~(x,y) \in \mathbb{R}^2 ~ | ~ y=0  \text{ and } u(x,0) \neq 0~\}$$
    that is,
    $$ E_i^0(\mathbb{R}) \cap U \simeq \{~x\in \mathbb{R} ~ | ~u(x,0) \neq 0~\}.$$ 
    By definition, one has 
    $$\widetilde{E}_i^{0,+}(\mathbb{R}) \cap U \simeq \{~(x,y,t) \in  (E_i^0(\mathbb{R}) \cap U) \times \mathbb{R} ~|~ u(x,y)t^{N_i} = 1 ~\} $$
    that is,
    $$\widetilde{E}_i^{0,+}(\mathbb{R}) \cap U \simeq \{~(x,t) \in  \mathbb{R}^2 ~|~ u(x,0)t^{N_i} = 1 ~\}.$$
    The projection onto the first factor $\widetilde{E}_i^{0,+}(\mathbb{R}) \cap U \to E_i^0(\mathbb{R}) \cap U$ is then a locally trivial covering for the Euclidean topology. After gluing, one obtains a covering $\widetilde{E}_i^{0,+}(\mathbb{R}) \to E_i^0(\mathbb{R})$ that is locally trivial for the Euclidean topology, but in general not locally trivial for the Zariski topology. In particular, there is no obvious relation between $[\widetilde{E}_i^{0,+}(\mathbb{R})]$ and $[E_i^0(\mathbb{R})]$ in the Grothendieck ring K$_0(\mathbb{R}\text{Var})$. Since the virtual Poincaré polynomial is not a topological invariant, there is also no obvious connection between $\beta(\widetilde{E}_i^{0,+}(\mathbb{R}))$ and $\beta(E_i^0(\mathbb{R}))$.\\
    Note also that the base $E_i^0(\mathbb{R})$ is generally not connected, so that the fiber of this covering is typically not constant and depends on the parity of $N_i$ and on the sign of $f$ in a neighborhood of the connected components of $E_i^0(\mathbb{R})$. More precisely:\\
    Assume first that $N_i$ is odd. Then the projection 
    $$(x,t) \in \widetilde{E}_i^{0,+}(\mathbb{R})\cap U \mapsto x \in E_i^0(\mathbb{R}) \cap U$$ is a regular homeomorphism, with inverse given by $x \mapsto (x, \sqrt[N_i]{u(x,0)^{-1}}~)$. Since $\beta$ is invariant under regular homeomorphisms (see \cite{mccrory2} Proposition 4.3), one has $\beta(\widetilde{E}_i^{0,+}(\mathbb{R}) \cap U) = \beta(E_i^0(\mathbb{R}) \cap U)$. Since $E_i^0(\mathbb{R})$ is covered by two open sets $U$, additivity yields $\beta(\widetilde{E}_i^{0,+}(\mathbb{R})) = \beta(E_i^0(\mathbb{R}))$.\\
    Now assume that $N_i$ is even. Then the projection $\widetilde{E}_i^{0,+}(\mathbb{R})\cap U \to  E_i^0(\mathbb{R}) \cap U$ is a covering of degree 0 (resp. degree 2) on the connected components of $E_i^0(\mathbb{R}) \cap U$ on which $u(x,0) < 0$ (resp. $u(x,0) >0$), that is, over the connected components of in whose neighborhood $f \circ \sigma$ is negative (resp. positive). Thus, when $N_i$ is even, there is in general no global trivialization of the covering $\widetilde{E}_i^{0,+}(\mathbb{R})\cap U \to  E_i^0(\mathbb{R}) \cap U$ nor, a fortiori, of the covering $\widetilde{E}_i^{0,+}(\mathbb{R}) \to  E_i^0(\mathbb{R})$ , even when $f\circ \sigma$ is positive in the neighborhood of $E_i^0(\mathbb{R}) \cap U$. Therefore, one cannot directly deduce the value of $\beta(\widetilde{E}_i^{0,+}(\mathbb{R}))$ from that of $\beta(E_i^0(\mathbb{R}))$.
\end{rem}

\begin{prop}
\label{prop 6}
    Assume that  $E_i$ satisfies $(E_i \cdot \sum\limits_{j \neq i} E_j) <3$. Then $\mathcal{R}_{top,i}^+ = 0$. Moreover, at the level of the virtual Poincaré polynomial, one also has $\mathcal{R}_{\beta,i}^+(u)= 0$.
\end{prop}

\begin{proof}
    Let us first assume that $(E_i \cdot \sum\limits_{j \neq i } E_j) = 1$, that is, $E_i$ intersects another component $E_1$ at a real point, and one has $\alpha_1 = -1$ by Proposition \ref{prop 3}. Assume first that $N_i$ is odd. As observed in the remark above, one has $\beta(\widetilde{E}_i^{0,+} (\mathbb{R})) = \beta(E_i^0(\mathbb{R}))=u$,  and $\widetilde{E}_{i,1}^{0,+}$ consists of a single point, since $ \gcd(N_i, N_1)$ is odd. It follows that 
    $$N_i\mathcal{R}_{\beta,i}^+(u) =u + \frac{u-1}{u^{\alpha_1}-1} = \frac{u^{\alpha_1  + 1}-1}{u^{\alpha_1}-1} = N_i\mathcal{R}_{\beta,i}(u) =  0.$$
    Now assume that $N_i$ is even. Since $E_i \simeq \mathbb{P}^1_{\mathbb{R}}$, we may assume, after an affine change of coordinates, that the intersection point $E_1 \cap E_i$ lies at infinity. Thus, $E_i^0$ is contained in a single affine chart $U \simeq \mathbb{A}^2_{\mathbb{R}}$ in which $E_i$ does not intersect any component, neither at real nor at complex points. On $U$, one can therefore write $f(\sigma(x,y)) = y^{N_i}u(x,y)$, so that 
    $$E_i^0 = E_i^0 \cap U = E_i \cap U  = \{~y  =0~\} \simeq \mathbb{A}^1_{\mathbb{R}}$$ 
    and $u$ is a unit, that is, $u$ does not vanish on $E_i^0 \cap U$. Equivalently, the polynomial $u(x,0) \in \mathbb{R}[x]$ has no real or complex roots, hence $u(x,0)$ is equal to a constant $\lambda \in \mathbb{R}$. Since $i \in J_{\mathbb{R}}^+$, we know that $E_i^0(\mathbb{R})$ intersects $\overline{P(f)}$, and therefore $\lambda > 0$. By definition, 
    $$\widetilde{E}_i^{0,+}(\mathbb{R}) = (\widetilde{E}_i^{0,+} \cap U )(\mathbb{R}) \simeq \{~(x,t) \in  \mathbb{R}^2 ~|~ u(x,0)t^{N_i} = 1 ~\}$$
    that is,
    $$\widetilde{E}_i^{0,+}(\mathbb{R}) \simeq \{~(x,t) \in  \mathbb{R}^2 ~|~ \lambda t^{N_i} = 1 ~\} = \{~(x,t) \in  \mathbb{R}^2 ~|~ t = \pm \sqrt[N_i]{\lambda^{-1}} ~\} \simeq \mathbb{R} \sqcup \mathbb{R}.$$
    By additivity, it follows that $\beta(\widetilde{E}_i^{0,+}(\mathbb{R}))= 2u$. In a neighborhood of the intersection point $E_i \cap E_1$, the situation is as follows

    \begin{figure}[H]
        \centering
        \includegraphics[scale=0.35]{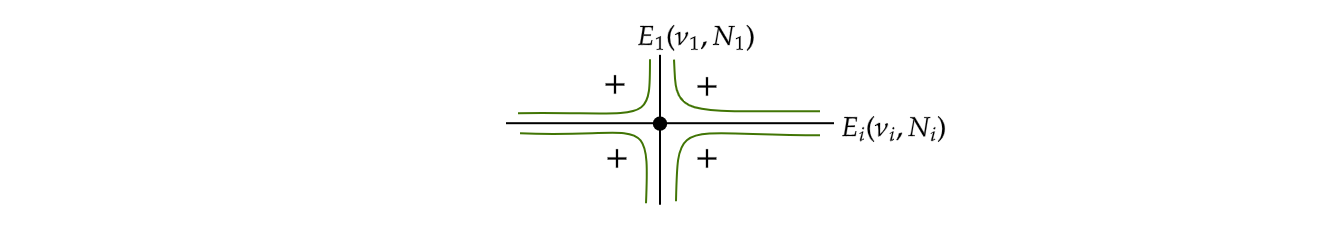}
    \end{figure}
and $\widetilde{E}_{i,1}^{0,+}$ consists of two points. It follows that 
$$N_i\mathcal{R}_{\beta,i}^+(u) = 2u + 2\frac{u-1}{u^{\alpha_1}-1} = 2\frac{u^{\alpha_1  + 1}-1}{u^{\alpha_1}-1} = 2N_i\mathcal{R}_{\beta,i}(u) = 0.$$

    Let us now assume that $(E_i \cdot \sum\limits_{j \neq i } E_j) = 2$, that is, $E_i$ intersects other components $E_1, E_2$ at two real points. By Proposition \ref{prop 3}, one has $\alpha_1 + \alpha_2 = 0$. Assume first that $N_i$ is odd. As seen above, one has $\beta(\widetilde{E}_i^{0,+}(\mathbb{R})) = \beta(E_i^0(\mathbb{R}))=u-1$ and both $\widetilde{E}_{i,1}^{0,+}, \widetilde{E}_{i,2}^{0,+}$ consist of a single point. Therefore
    $$N_i\mathcal{R}_{\beta,i}^+(u) =  u-1 + \frac{u-1}{u^{\alpha_1} -1} + \frac{u-1}{u^{\alpha_2}-1} = N_i\mathcal{R}_{\beta,i}(u) = 0.$$
    
    Suppose now that $N_i$ is even. As before, we may assume that the intersection point $E_i \cap E_2$ lies at infinity, so that $E_i^0$ is contained in a single affine chart $U \simeq \mathbb{A}^2_{\mathbb{R}}$ in which $E_i$ intersects the component $E_1$ at a real point. By performing a translation, we may further assume that $E_i \cap E_1$ is the origin. On $U$, one can then write $f(\sigma(x,y)) = y^{N_i}u(x,y)$ where 
    $$E_i^0 = E_i^0 \cap U = (E_i \setminus E_i \cap E_1) \cap U  = \{~y  =0 \text{ and } x \neq 0~\} \simeq \mathbb{A}^1_{\mathbb{R}}\setminus \{ 0 \}$$
    and where $u$ does not vanish on $E_i^0 \cap U$. Equivalently, the polynomial $u(x,0) \in \mathbb{R}[x]$ vanishes only at the origin (including complex roots). Therefore, $u(x,0)$ is of the form  $\lambda x^{N_1}$ for some $\lambda  \in \mathbb{R}^*$. By definition,
    $$\widetilde{E}_i^{0,+}(\mathbb{R}) \simeq \{~(x,t) \in  \mathbb{R}^2 ~|~ \lambda  x^{N_1}t^{N_i} = 1 ~\}.$$
    If $N_1$ is odd, the projection $(x,t) \in \widetilde{E}_i^{0,+}(\mathbb{R}) \mapsto t \in \mathbb{R}^*$ is a regular homeomorphism, with inverse given by $t \mapsto (\sqrt[N_1]{(\lambda t^{N_i})^{-1}}~ , t ) $. Hence, $\beta(\widetilde{E}_i^{0,+}(\mathbb{R})) = \beta(\mathbb{R}^*) = u-1$. The covering $\widetilde{E}_i^{0,+}(\mathbb{R}) \to E_i^0(\mathbb{R})$ is then as follows
    
    \begin{figure}[H]
        \centering
        \includegraphics[scale=0.35]{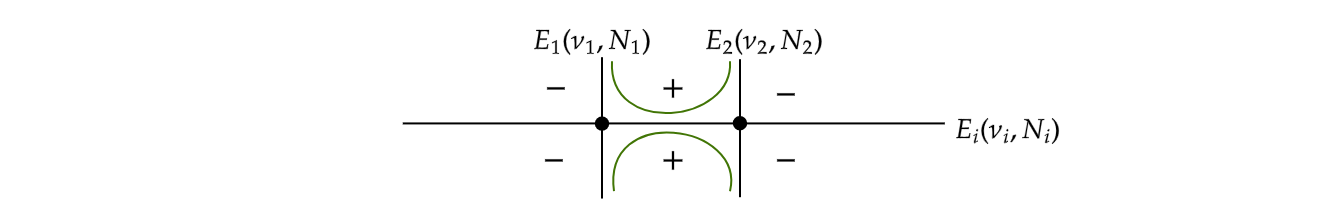}
    \end{figure}

    and $\widetilde{E}_{i,1}^{0,+}, \widetilde{E}_{i,2}^{0,+}$ both consist of a single point. It follows that

    $$N_i\mathcal{R}_{\beta,i}^+(u) =  u-1 + \frac{u-1}{u^{\alpha_1} -1} + \frac{u-1}{u^{\alpha_2}-1} = N_i\mathcal{R}_{\beta,i}(u) = 0.$$

    If $N_1$ is even, then $\lambda \in \mathbb{R}^*_+$ because $i \in J_{\mathbb{R}}^+$, and one has $\beta(\widetilde{E}_i^{0,+}(\mathbb{R})) = 2(u-1)$ thanks to the following Lemma \ref{lem 2}. The covering $\widetilde{E}_i^{0,+}(\mathbb{R}) \to E_i^0(\mathbb{R})$ is as follows
    
    \begin{figure}[H]
        \centering
        \includegraphics[scale=0.35]{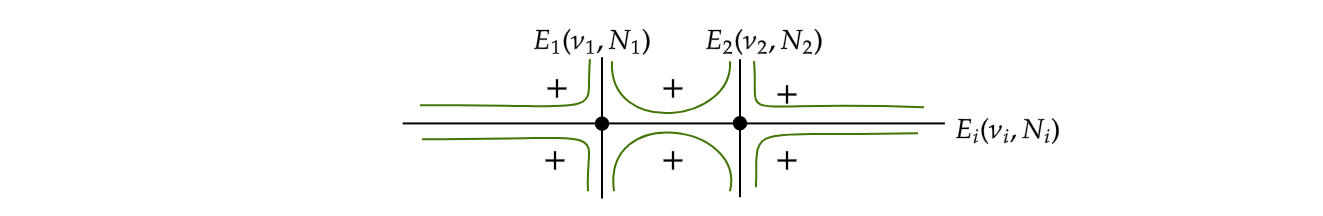}
    \end{figure}

    so that $\widetilde{E}_{i,1}^{0,+}, \widetilde{E}_{i,2}^{0,+}$ both consist of two points. Therefore,

    $$N_i\mathcal{R}_{\beta,i}^+(u) =  2(u-1) + 2\frac{u-1}{u^{\alpha_1} -1} + 2\frac{u-1}{u^{\alpha_2}-1} = 2N_i\mathcal{R}_{\beta,i}(u) = 0.$$
    
    Finally, regarding the contribution at the level of topological zeta functions, we know that $\mathcal{R}_{\beta,i}^+(u)$ can be seen as a $\mathcal{C}^\infty$ function of the variable $u \in \mathbb{R}^*_+\setminus\{1\}$, which can be continuously extended to 1 with
    $$\mathcal{R}_{top,i}^+= \lim_{u \to 1} \mathcal{R}_{\beta,i}^+(u)  = 0.$$
    
\end{proof}

\begin{lem}
    \label{lem 2}
    For all $m,p \in \mathbb{N}^*$, one has
    $$\beta(\{~(x,t) \in  \mathbb{R}^2 ~|~ x^{2m}t^{2p} = 1 ~\} ) = 2(u-1).$$
\end{lem}

\begin{proof}
    We can proceed by induction on $m$. For $m=1$, one has
    $$\{~(x,t) \in  \mathbb{R}^2 ~|~ x^2t^{2p} = 1 ~\} = \{~(x,t) \in  \mathbb{R}^2 ~|~ xt^{p} = 1 ~\} \sqcup \{~(x,t) \in  \mathbb{R}^2 ~|~ xt^{p} =- 1 ~\} $$
    which is isomorphic to two disjoint copies of $\mathbb{R}^*$, so the result holds in this case. For $ m \geq 1$, one can similarly decompose 
    $$\{~(x,t) \in  \mathbb{R}^2 ~|~ x^{2m}t^{2p} = 1 ~\} = \{~(x,t) \in  \mathbb{R}^2 ~|~ x^{m}t^{p} = 1 ~\} \sqcup \{~(x,t) \in  \mathbb{R}^2 ~|~ x^mt^{p} = -1 ~\}.$$
    If $m$ (resp. $p$) is odd, the projection onto the first coordinate (resp. onto the second coordinate) gives a regular homeomorphism 
    $$\{~(x,t) \in  \mathbb{R}^2 ~|~ x^mt^{p} = \pm1 ~\} \simeq \mathbb{R}^*$$ 
    and the desired result follows.\\
    If both $m$ and $p$ are even, the set $\{~(x,t) \in  \mathbb{R}^2 ~|~ x^mt^{p} = -1 ~\}$ is empty, so that
    $$\{~(x,t) \in  \mathbb{R}^2 ~|~ x^{2m}t^{2p} = 1 ~\} =  \{~(x,t) \in  \mathbb{R}^2 ~|~ x^mt^{p} = 1 ~\}$$
    which allows us to conclude by induction.
\end{proof}

As in the proof of Corollary \ref{cor 4}, one can work in the ring K$_0(\mathbb{R}\text{Var})$ and verify that the contributions in Proposition \ref{prop 6} are also zero in the motivic setting.\\

The corollary below follows immediately from Proposition \ref{prop 6} and Theorem \ref{thm 4}.

\begin{cor}
\label{coro 5}
    Every pole of the positive topological zeta function is also a pole of the naive topological zeta function, that is, one has the inclusion
    $$\text{Poles}(Z_{top,0}^+(f;s)) \subset \text{Poles}(Z_{top,0}(f;s))  \cap \{~ -\frac{\nu_i}{N_i} \mid  i \in J_{\mathbb{R}}^{+}~\}$$
    and similarly, 
    $$\text{Poles}(Z_{top,0}^-(f;s)) \subset \text{Poles}(Z_{top,0}(f;s))  \cap \{~ -\frac{\nu_i}{N_i} \mid  i \in J_{\mathbb{R}}^{-}~\}.$$
    Furthermore, the same inclusions holds for the motivic zeta functions, as well as for the zeta functions defined at the level of the virtual Poincaré polynomial.
    
\end{cor}

\begin{thm}
\label{thm 6}
     Assume that $E_i$ is an exceptional curve such that the multiplicity $N_i$ is odd. Then the contribution $\mathcal{R}_{top,i}^+$ is nonzero if and only if $(E_i \cdot \sum\limits_{j \neq i} E_j) \geq 3$. In this case, one has:
    \begin{enumerate}
        \item $\mathcal{R}_i^+ > 0$ if and only if $\alpha_j > 0$ for all $j \in \llbracket 1, k \rrbracket$, and we furthermore have $\mathcal{R}_i^+ \geq \frac{2}{N_i}$.
        \item $\mathcal{R}_i^+ < 0 $ if and only if there exists $j \in \llbracket 1, k \rrbracket$ such that $\alpha_j < 0$.   
    \end{enumerate}
    Moreover, the contribution $\mathcal{R}_{\beta,i}^+$ at the level of the virtual Poincaré polynomial is also nonzero if and only if $(E_i \cdot \sum\limits_{j \neq i} E_j) \geq 3$.  
\end{thm}

\begin{proof}
Since $N_i$ is odd, we have seen in Remark \ref{rem 12} that $\beta(\widetilde{E}_i^{0,+}(\mathbb{R})) = \beta(E_i^0(\mathbb{R}))$. Moreover, each terms $\widetilde{E}_{i,j}^{0,+}(\mathbb{R})$ consists of a single point because $\gcd(N_i, N_j)$ is odd. Hence, one has
 $\mathcal{R}_{\beta,i}^+=\mathcal{R}_{\beta,i}$ and the theorem now follows from Theorem \ref{thm 3}.
\end{proof}

Let us now assume that $N_i$ is even and denote $C = \widetilde{E}_i^{0,+}(\mathbb{R})$. By construction, $C$ is a smooth (this can be verified locally) real algebraic curve, however, $C$ is not compact in general. In fact, one can check without much difficulty that $C$ is compact if and only if $E_i$ does not intersect any component at a real point.\\
Let us consider $C \hookrightarrow \widetilde{C}$, a smooth compactification of $C$. By additivity of the virtual Poincaré polynomial, one has 
$$\beta(C) = \beta(\widetilde{C}) - \beta(\widetilde{C} \setminus C ).$$
Since $\widetilde{C}$ is a smooth compact curve, $\widetilde{C}$ is topologically a union of circles. Let $c$ be the number of these circles, that is, the number of connected components of $\widetilde{C}$. Then one has 
$$\beta(C) = c (u+1) - \beta(\widetilde{C} \setminus C ) = c(u+1) + \chi_c(C).$$

A smooth compactification of $C$ can be described quite explicitly. Let us first describe a compactification $C \hookrightarrow \overline{C}$. Locally, let $U\simeq \mathbb{R}^2$ be a Zariski open set on which $f(\sigma(x,y)) = u(x,y)y^{N_i}$, and denote $P(x) = u(x,0) \in \mathbb{R}[x]$, so that 
$$C \cap U \simeq \{~(x,t) \in  \mathbb{R}^2 ~|~ P(x)t^{N_i} = 1 ~\}.$$
Locally as above, the sign of $f \circ \sigma$ is independent of $t$ and is determined by the sign of $P$.
The closure of $C \cap U$ in $(E_i \cap U) \times \mathbb{P}^1(\mathbb{R}) \simeq \mathbb{R} \times \mathbb{P}^1(\mathbb{R})$ is the algebraic curve 
$$\{~(x,[t;y]) \in  \mathbb{R} \times \mathbb{P}^1(\mathbb{R}) ~|~ P(x)t^{N_i} = y^{N_i} ~\}$$
 where the points at infinity correspond to the real roots of $P$. One can then glue the curves above along the open sets $E_i \cap U$ to obtain $\overline{C}$, and it follows that, globally, $\overline{C} \setminus C$ corresponds to the number of real intersection points of $E_i$ with other components.
Note that the above curve is in fact contained in a single affine chart corresponding to $t=1$, so that 
$$\overline{C} \cap U \simeq \{~(x,y) \in  \mathbb{R}^2 ~|~ P(x) = y^{N_i} ~\}.$$
When $P$ has simple roots, the curve $\overline{C}$ is smooth and of hyperelliptic type (apart from the fact that $N_i$ is generally strictly greater than 2). For example, the real locus of $y^8-x(x-1)(x-2)(x-\frac{1}{2})(x+2)(x+4)$ is as follows.
     
    \begin{figure}[H]
        \centering
        \includegraphics[scale=0.35]{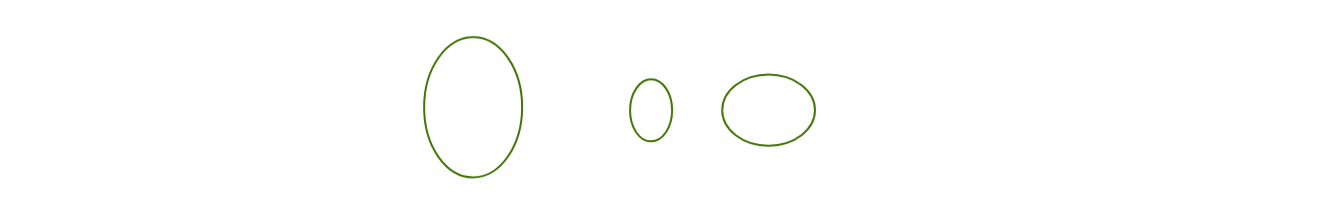}
    \end{figure}

Moreover, the number of connected components of $\overline{C}$ in this case is equal to $\frac{k}{2}$, where $k$ is the number of real intersection points of $E_i$ with other components (one sees that $k$ is necessarily even by the first equality in the proof of Proposition \ref{prop 3}).\\
Let us now consider the general case, where we no longer assume that $P$ has simple roots. The singular points of $\overline{C}$ then correspond (locally as above) to the roots of $P$ with multiplicity strictly greater than $1$. Consider, for example, the curve $\overline{C}$ defined locally by the equation $y^{8}+x(x-1)(x-3)^{6}(x+2)^{3}(x+4)^2(x-5)$ and denote by $\nu : \widetilde{C} \to \overline{C}$ its normalization. The real locus of these curves is shown below, where we have depicted the roots of $P$ as points of $\overline{C}$, as well as their preimages under $\nu.$

    \begin{figure}[H]
        \centering
        \includegraphics[scale=0.35]{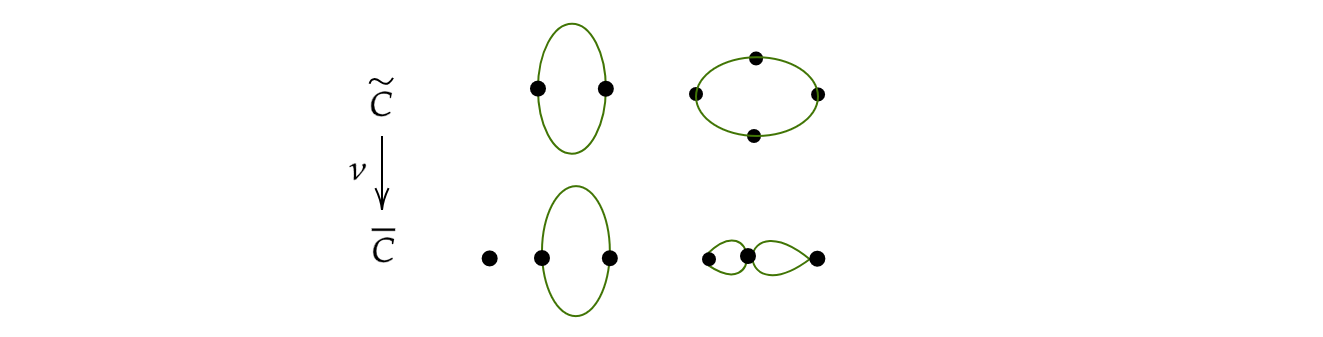}
    \end{figure}

\begin{prop}
    \label{prop 7}
    With the above notations, one has
    $$-\chi_c(C) = \beta(\widetilde{C} \setminus C )  = \sum\limits_{j=1}^k \beta(\widetilde{E}_{i,j}^{0,+}(\mathbb{R})).$$
\end{prop}

\begin{proof}
    We will count the number of points in $\widetilde{C}\setminus C$. Recall that for $j\in \llbracket 1, k \rrbracket$, $\beta(\widetilde{E}_{i,j}^{0,+}(\mathbb{R}))$ is equal to $1$ when $N_j$ is odd and equal to $0$ or $2$ when $N_j$ is even, depending on the sign of $f \circ \sigma$ in a neighborhood of $E_i \cap E_j$. One can write 
    $$\overline{C} = C \sqcup \overline{C}\setminus C$$
    where we have seen above that the points at infinity, i.e., the set $\overline{C} \setminus C$, corresponds to the real intersection points of $E_i$ with other components. The normalization $\nu : \widetilde{C} \to \overline{C}$ is an isomorphism outside the singular locus of $\overline{C}$, in particular, it is an isomorphism on $C$. One can then write
    $$\widetilde{C} = \nu^{-1}(\overline{C}) = \nu^{-1}(C) \sqcup \nu^{-1}(\overline{C} \setminus C)$$
    where $\nu^{-1}(C) \simeq C$ and $\widetilde{C} \setminus C \simeq \nu^{-1}(\overline{C} \setminus C)$. It remains to count the number of preimages of each point of $\overline{C}\setminus C$ under $\nu$, which can be done locally. Let $p \in \overline{C} \setminus C$, corresponding to a real intersection point of $E_i$ with another component $E_j$. Using the notation from the discussion above, let $U$ be a Zariski open set containing $p=E_i \cap E_j$, on which 
    $$\overline{C} \cap U \simeq \{~(x,y) \in  \mathbb{R}^2 ~|~ P(x) = y^{N_i} ~\}$$
    and where $p$ corresponds to a real root of $P$. By performing a translation, we may assume that $p$ is the origin in $\mathbb{R}^2$. In a neighborhood of the origin, the curve $\overline{C}$ is then analytically equivalent to a curve with equation $y^{N_i} = \pm x^{N_j}$, whose normalization is well-known in the different cases.\\
    If $N_j$ is odd, then $\nu^{-1}(p)$ consists of a single point, as does $\widetilde{E}_{i,j}^{0,+}(\mathbb{R})$.\\
    Now, if $N_j$ is even and $\overline{C}$ is analytically equivalent to $y^{N_i} = x^{N_j}$ in a neighborhood of the origin, then $\nu^{-1}(p)$ consists of two points, and $f \circ \sigma$ is positive in a neighborhood of $p$, so that $\widetilde{E}_{i,j}^{0,+}(\mathbb{R})$ also consists of two points.\\
    Finally, assume that $N_j$ is even and that $\overline{C}$ is analytically equivalent to $y^{N_i} = -x^{N_j}$ in a neighborhood of the origin. In this case, $f \circ \sigma$ is negative in a neighborhood of $p$, so that $\widetilde{E}_{i,j}^{0,+}(\mathbb{R})$ is empty. Moreover, the real locus of $\overline{C}$ in a neighborhood of the origin consists of a single isolated point, hence $\nu^{-1}(p) = \emptyset$.
\end{proof}

\begin{cor}
    Assume that $(E_i \cdot \sum\limits_{ j \neq i } E_j) \geq 3$ and that $\alpha_j > 0$ for all $j \in \llbracket 1, k \rrbracket$ such that $E_i(\mathbb{R}) \cap E_j(\mathbb{R}) \cap \overline{P(f)} \neq \emptyset$ (that is $\mu(\widetilde{E}_{i,j}^{0,+}(\mathbb{R})) \neq 0$). Then $\mathcal{R}_{top,i}^+ \geq \frac{2c}{N_i} > 0 $.
\end{cor}

\begin{proof}
    By Proposition \ref{prop 7}, one can write
    $$\mathcal{R}_{top,i}^+ = \frac{1}{N_i} \left( 2c + \sum\limits_{j=1}^k \mu(\widetilde{E}_{i,j}^{0,+}(\mathbb{R}))( \frac{1}{\alpha_j} -1) \right) \geq \frac{2c}{N_i}.$$ 
    
\end{proof}

\begin{rem}
    In general, the existence of $j \in \llbracket 1, k \rrbracket$ such that $\alpha_j < 0$ and $\mu(\widetilde{E}_{i,j}^{0,+}(\mathbb{R})) \neq 0$ does not necessarily imply that $\mathcal{R}_{top,i}^+ < 0$, unlike in the naive case. There may even be cancellations, so that the contribution at the topological level can be zero, as illustrated by the following example.
\end{rem}

\begin{ex}
    \label{ex 6}
    Consider the homogeneous polynomial $f = xy(x-y)^3(x-2y)^7$. An embedded resolution of $f$ can be obtained by blowing-up the origin, and the graph of the resolution, together with the coverings $\widetilde{E}_i^{0,+}(\mathbb{R})$, is as follows.

    \begin{figure}[H]
        \centering
        \includegraphics[scale=0.35]{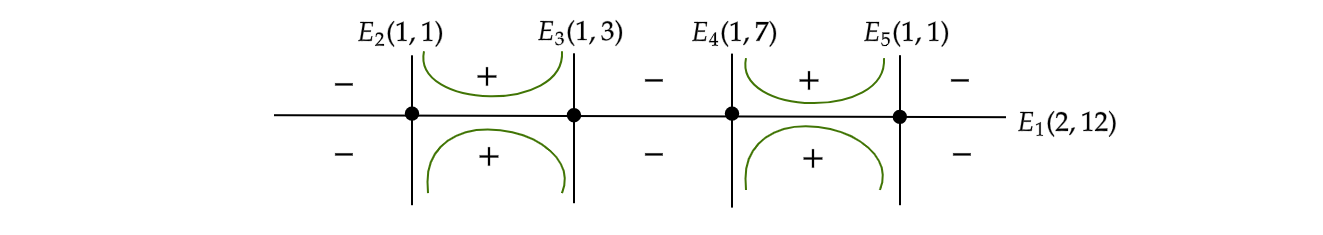}
    \end{figure}
    
    A smooth compactification of $\widetilde{E}_1^{0,+}(\mathbb{R})$ has two connected components, and one finds that 
    $$\mathcal{R}_{top,1}^+ = \sum\limits_{j = 2}^5\frac{1}{\alpha_j} = \frac{6} {5} +2 - 6 + \frac{6}{5} < 0.$$
    Now, if one consider $f =xy(x-y)^3(x-2y)^9$, the resolution graph is the same but the multiplicities $N_4$ and $N_1$ are now equal to $9$ and $14$ respectively. It follows that 
    $$\mathcal{R}_{top,1}^+ = \sum\limits_{j = 2}^5\frac{1}{\alpha_j} = \frac{7}{6} + \frac{7}{4} -\frac{7}{2} + \frac{7}{6} > 0.$$
    Finally, if one consider $f = xy(x-y)(x-2y)^5$, the graph of the solution is the same, but the multiplicities $N_1, N_2, N_3, N_4, N_5$ are $8, 1, 1, 5, 1$, respectively, so that 
    $$\mathcal{R}_{top,1}^+ = \sum\limits_{j = 2}^5\frac{1}{\alpha_j}= \frac{4}{3} + \frac{4}{3} + \frac{4}{3} - 4 = 0.$$
    In particular, the poles of $Z_{top,0}^+(f;s)$ are $-1,-\frac{1}{5}$, while the poles of $Z_{top,0}(f;s)$ are $-1,-\frac{1}{4},-\frac{1}{5}$, and one has the strict inclusion 
    $$\text{Poles}(Z_{top,0}^+(f;s)) \subsetneq \text{Poles}(Z_{top,0}(f;s))  \cap \{~ -\frac{\nu_i}{N_i} \mid  i \in J_{\mathbb{R}}^{+}~\}.$$
    However, for the contribution at the level of the virtual Poincaré polynomial, one finds 
    $$\mathcal{R}_{\beta,1}^+(u) = 2(u-1)+3\frac{u-1}{u^{\frac{3}{4}}-1} + \frac{u-1}{u^{-\frac{1}{4}}-1} = \frac{(2u^{\frac{1}{2}} -u^{\frac{3}{4}}  + u^{-\frac{1}{4}}-2)(u-1) }{(u^{\frac{3}{4}}-1)(u^{-\frac{1}{4}}-1)  } $$
    which is not identically zero, even though $\underset{u \to 1}{\lim}\mathcal{R}_{\beta,1}^+(u)  = \mathcal{R}_{top,1}^+ = 0$. In particular, $-\frac{1}{4}$ is a pole of $Z_{\beta,0}(f;u^{-s})$, and one also has a strict inclusion 
    $$\text{Poles}(Z_{top,0}^+(f;s)) \subsetneq \text{Poles}(Z_{\beta,0}^+(f;u^{-s})) $$
    unlike what happens in the naive case.

\end{ex}

\begin{prop}
\label{prop 11}
    At the level of the virtual Poincaré polynomial, the contribution $\mathcal{R}_{\beta,i}^+$ is nonzero if and only if $(E_i \cdot \sum\limits_{j \neq i} E_j) \geq 3.$
\end{prop}

\begin{proof}
    First, recall that the contribution at the level of the virtual Poincaré polynomial is given by 
    $$N_i\mathcal{R}_{\beta,i}^+(u) = \beta(\widetilde{E}_i^{0,+}(\mathbb{R})) + \sum\limits_{j=1}^k  \beta(\widetilde{E}_{i,j}^{0,+}(\mathbb{R}))\frac{u-1}{u^{\alpha_j}-1}$$
    and that, by Proposition \ref{prop 7}, one has $\beta(\widetilde{E}_i^{0,+}(\mathbb{R})) = c (u+1) -\sum\limits_{ j =1}^k \beta(\widetilde{E}_{i,j}^{0,+}(\mathbb{R}))$, where $c$ is the number of connected components of a smooth compactification of $\widetilde{E}_i^{0,+}(\mathbb{R})$. Therefore, one has 
    $$N_i\mathcal{R}_{\beta,i}^+(u) = c(u+1) + \sum\limits_{j=1}^k  \beta(\widetilde{E}_{i,j}^{0,+}(\mathbb{R}))\frac{u-u^{\alpha_j}}{u^{\alpha_j}-1}.$$
    If $ \alpha_j > 0$ for all $j \in \llbracket 1, k \rrbracket$ such that $\beta(\widetilde{E}_{i,j}^{0,+}(\mathbb{R})) \neq 0$, we have seen in the previous corollary that $\underset{u  \to 1 }{\lim}\mathcal{R}_{\beta,i}^+(u) = \mathcal{R}_{top,i}^+ \neq 0$ and in particular, it follows that $\mathcal{R}_{\beta,i}^+$ is not identically zero. Now, assume that there exists some $\alpha_1 < 0 $ such that $\beta(\widetilde{E}_{i,1}^{0,+}(\mathbb{R})) \neq 0$. To prove that $\mathcal{R}_{\beta,i}^+$ is not identically zero, it suffices to show that the leading terms of the asymptotic expansion of $\mathcal{R}_{\beta,i}^+ $ at $0$ are non-zero. One finds that
    $$\frac{u - u^{\alpha_1}}{u^{\alpha_1}-1}= \frac{u^{1-\alpha_1}-1}{1-u^{-\alpha_1}} = -1 - u^{-\alpha_1} + \underset{u \to 0}{o}(u^{-\alpha_1}).$$
    On the other hand, since $(E_i \cdot \sum\limits_{j \neq i } E_j) \geq 3$, one knows that $-\alpha_1 < \underset{2 \leq j \leq k }{\min} \alpha_j$ by Corollary \ref{cor 1}, and for $j \in \llbracket 1, k \rrbracket$ one has 
    $$\frac{u-u^{\alpha_j}}{u^{\alpha_j}-1} = u^{\alpha_j} + \underset{u \to 0}{o}(u^{\alpha_j}) = \underset{u \to 0}{o}(u^{-\alpha_1}).$$
    Therefore
    $$N_i\mathcal{R}_{\beta,i}^+(u) = c - \beta(\widetilde{E}_{i,1}^{0,+}(\mathbb{R})) -\beta(\widetilde{E}_{i,1}^{0,+}(\mathbb{R}))u^{-\alpha_1} + \underset{u \to 0}{o}(u^{-\alpha_1})$$
    which concludes the proof.
\end{proof}

The previous theorem allows us to describe the poles of $Z_{\beta,0}^+(f;u^{-s})$ in cases where there is at most one nonzero contribution for a given candidate pole (this happens, for example, when $f=g^r$ with $g$ analytically irreducible).

\begin{cor}
\label{cor 6}
    Let $s_0 \in \mathbb{Q}$. Assume there exists exactly one $i \in J_{\mathbb{R}}^{\pm}$ such that either $E_i$ is an exceptional curve satisfying $(E_i\cdot \sum\limits_{j\neq i } E_j) \geq 3$ and $s_0=-\frac{\nu_i}{N_i}$ or such that $E_i$ is an irreducible component of the strict transform and $s_0=-\frac{1}{N_i}$. Then $s_0$ is a pole of $Z_{\beta,0}^{\pm}(f;u^{-s})$.\\
    Equivalently, in cases where there is at most one nonzero contribution for every candidate pole as above, Proposition \ref{prop 11}, together with the results of section \ref{section 2}, yields
    $$\text{Poles}(Z_{\beta,0}^{\pm}(f;u^{-s}))= \text{Poles}(Z_{\beta,0}(f;u^{-s}))  \cap \{~ -\frac{\nu_i}{N_i} \mid  i \in J_{\mathbb{R}}^{\pm}~\}.$$
\end{cor}

\begin{rem}\begin{enumerate}
    \item We do not known whether the above equality always holds, or whether cancellations between different nonzero contributions can occur.
    \item The data of the poles of naive zeta functions and zeta functions with signs alone is generally not sufficient to distinguish two germs that are not blow-Nash equivalent, even for curves. For example, if $\varepsilon \in \{ \pm1 \}$ and $f^{\varepsilon} = x^3  + \varepsilon y^4$, one finds that 
    $$\text{Poles}(Z_{top,0}^{\pm}(f^{\varepsilon};s)) = \text{Poles}(Z_{top,0}(f^{\varepsilon};s))  = \{ -1, -\frac{7}{12} \}$$
    although $f^+$ and $f^-$ are not blow-Nash equivalent, as one can check by looking at the Fukui invariants with signs of these two functions.

    \end{enumerate}
\end{rem}

\section{Poles and monodromy eigenvalues}
\label{section 4}
In connection with the monodromy conjecture, it is natural to seek an interpretation of the poles of these real zeta functions in terms of monodromy and its eigenvalues. We briefly present the objects appearing in the conjecture and refer to \cite{veys3} for a detailed introduction.\\
For now, let us take $f : \mathbb{C}^{d} \to \mathbb{C}$ a non-constant polynomial map sending $0$ to $0$, denote by $V \subset \mathbb{C}^{d}$ the hypersurface defined by $f$ and let $a \in V$.

\begin{prop}[\cite{milnor1}, \cite{veys3} Proposition 2.5]
\label{prop 8}
Denote $P_{a,i}(t)$ the characteristic polynomial of the monodromy $T^*$ acting on $H^i(\mathcal{F}_{f,a};\mathbb{C})$. Then 
    \begin{enumerate}
        \item All monodromy eigenvalues are roots of unity.
        \item if $f = f_1^{M_1} \dots f_r^{M_r}$ is the decomposition of $f$ in irreducible components and $m = \underset{a \in \{ f_j = 0 \}}{gcd} M_j$, then $P_{a,0}(t) = t^m -1$.
        \item When $a$ is an isolated critical point of $\{f= 0 \}$, then $H^i(\mathcal{F}_{f,a};\mathbb{C}) = 0$ for $i \neq 0,d-1$. Moreover, $H^{d-1}(\mathcal{F}_{f,a};\mathbb{C}) \neq 0$ and $P_{a,0}(t) = t-1$.
        \item When $a$ is a smooth point of $\{ f =0\}$, then $H^i(\mathcal{F}_{f,a};\mathbb{C}) = 0$ for $i > 0$ and $P_{a,0}(t) = t-1$.
    \end{enumerate}
\end{prop}

 \begin{defi}
     The monodromy zeta function is defined by 
     $$\zeta_a(t) = \prod_{i\geq 0} P_{a,i}(t)^{(-1)^{i+1}}.$$
 \end{defi}

 \begin{thm}[\cite{acampo} Theorem 3]
 \label{thm 9}
     Let $\sigma : X \to \mathbb{C}^{d}$ be an embedded resolution of $f$. Then, using the usual notations, one has 
     $$\zeta_a(t) = \prod_{j \in J}(t^{N_j}-1)^{-\chi(E_j^0 \cap \sigma^{-1}(a))}$$
 \end{thm}

 Suppose now that $f \in \mathbb{C}[x,y]$, so that $H^*(\mathcal{F}_{f,a};\mathbb{C}) = H^0(\mathcal{F}_{f,a};\mathbb{C}) \bigoplus H^1(\mathcal{F}_{f,a};\mathbb{C})$ and $\zeta_a(t) = \frac{P_{a,1}(t)}{P_{a,0}(t)}$. The monodromy conjecture, which has been proved in the case of curves, then corresponds to the following theorem.

 \begin{thm}
    \label{thm 8}
     Let $s_0$ be a pole of $Z_{top,0}(f;s)$. Then $e^{2i\pi s_0}$ is an eigenvalue of the monodromy operator $T^* : H^*(\mathcal{F}_{f,a};\mathbb{C}) \to H^*(\mathcal{F}_{f,a};\mathbb{C})$ for some $a \in V$ close to the origin.
 \end{thm}

 In the above theorem, “close to the origin” means that the origin belongs to the Zariski closure of the set of points $a$ such that $e^{2i\pi s_0}$ is an eigenvalue of $T^* : H^*(\mathcal{F}_{f,a};\mathbb{C}) \to H^*(\mathcal{F}_{f,a};\mathbb{C})$.

\begin{rem}
It is necessary to consider eigenvalues of the monodromy acting on $H^*(\mathcal{F}_{f,a};\mathbb{C})$ for points $a$ that are not necessarily the origin but lie sufficiently close to it. For example, if $f=x^3y^4$, one computes 
$$Z_{top,0}(f;s) =\frac{1}{(1+3s)(1+4s)}.$$ 
However, by Proposition \ref{prop 8} and Theorem $\ref{thm 9}$ one has $P_{0,0}(t) = P_{0,1}(t) = t -1$, so that $e^{-\frac{2i\pi}{3}}$ and $e^{-\frac{2i\pi}{4}}$ are not eigenvalues of the monodromy $T^* :  H^*(\mathcal{F}_{f,O};\mathbb{C}) \to H^*(\mathcal{F}_{f,O};\mathbb{C})$. By contrast, for $a\in \{x = 0\} \setminus \{y=0 \}$, one finds $P_{a,0}(t)=t^3-1$, while for $a \in \{y = 0\} \setminus \{x=0 \}$, one finds $P_{a,0}(t)=t^4-1$.
\end{rem}

\begin{proof}[Proof of theorem \ref{thm 8}]
   One possible proof is the one presented in \cite{veys3}, which uses Veys' criterion (Theorem~\ref{thm 5}) as well as A’Campo’s formula (Theorem~\ref{thm 9}). Moreover, this proof shows that if the pole $s_0 = -\frac{\nu_i}{N_i}$ is induced by an exceptional curve, then $e^{-\frac{2i\pi \nu_i}{N_i}}$ is an eigenvalue of the monodromy $ T^* : H^*(\mathcal{F}_{f,O};\mathbb{C}) \to H^*(\mathcal{F}_{f,O};\mathbb{C})$ acting on the Milnor fiber of $f$ at the origin. If $s_0 = -\frac{1}{N_i}$ is induced by a component $E_i$ of the strict transform, then, as in the remark above, one shows that $e^{-\frac{2 i \pi}{N_i}}$ is an eigenvalue of the monodromy $T^* : H^0(\mathcal{F}_{f,a};\mathbb{C}) \to H^0(\mathcal{F}_{f,a};\mathbb{C})$ in degree $0$ for a suitably chosen point $a$ i.e. for a point $a$ lying exclusively on the branch $\{ f_i = 0 \}$ that induces the component $E_i$ of the strict transform.
\end{proof}

Let us now consider $f\in \mathbb{R}[x,y]$ vanishing at the origin, $\sigma :(X,\sigma^{-1}(0)) \to (\mathbb{A}^2_{\mathbb{R}},0)$ the canonical embedded resolution of $f$ and $V \subset \mathbb{A}^2_{\mathbb{R}}$ the curve defined by $f$. We have seen that

$$ \text{Poles}(Z_{top,0}(f;s)) \subset \text{Poles}(Z_{top,0}(f_{\mathbb{C}};s)) \cap  \{~ -\frac{\nu_i}{N_i} \mid  i \in J_{\mathbb{R}}~\}$$

and also that
$$\text{Poles}(Z_{top,0}^{\pm}(f;s)) \subset \text{Poles}(Z_{top,0}(f_{\mathbb{C}};s)) \cap  \{~ -\frac{\nu_i}{N_i} \mid  i \in J_{\mathbb{R}}^{\pm}~\}.$$

It is therefore natural to try to translate these inclusions in terms of the eigenvalues of the monodromy. In other words, one seeks a subset, say $E$, of the set of eigenvalues of the monodromy $T ^* : H^*(\mathcal{F}_{f,a};\mathbb{C}) \to H^*(\mathcal{F}_{f,a};\mathbb{C})$ for $a\in V(\mathbb{C})$ close to the origin, such that every pole $s_0$ of $Z_{top,0}(f;s)$ induces an eigenvalue $e^{2i \pi s_0} \in E$, and similarly for the poles of zeta functions with signs.

\begin{prop}
    \label{prop 9}
    Let $s_0$ be a pole of $Z_{top,0}(f;s)$. Then $e^{2i \pi s_0}$ is an eigenvalue of the monodromy $T^* : H^*(\mathcal{F}_{f,a};\mathbb{C}) \to H^*(\mathcal{F}_{f,a};\mathbb{C})$ for some $a \in V(\mathbb{R})$ close to the origin.
\end{prop}

\begin{proof}
    The proof is similar to that of Proposition \ref{prop 10} below.
\end{proof}

Let us now denote 
$$V_+(\mathbb{R}) = V (\mathbb{R}) \cap \overline{\{(x,y) \in \mathbb{R}^2 \mid f(x,y) > 0 \}} \text{ and }V_{-}(\mathbb{R}) = V(\mathbb{R}) \cap \overline{\{(x,y) \in \mathbb{R}^2 \mid f(x,y) < 0 \}}.$$ 

\begin{prop}
\label{prop 10}
    Let $s_0$ be a pole of $Z_{top,0}^{\pm}(f;s)$. Then $e^{2i \pi s_0}$ is an eigenvalue of the monodromy $T^* : H^*(\mathcal{F}_{f,a};\mathbb{C}) \to H^*(\mathcal{F}_{f,a};\mathbb{C})$ for some $a \in V_{\pm}(\mathbb{R})$ close to the origin.
\end{prop}

\begin{proof}
     By symmetry, we only consider the case of the positive zeta function and thus let $s_0$ be a pole of $Z_{top,0}^{+}(f;s)$. Write the decomposition into irreducibles $f=f_1^{M_1} \dots f_r^{M_r}$ and distinguish two cases. First, suppose that $s_0=-\frac{\nu_i}{N_i} $ for a certain exceptional curve $E_i$ such that $E_i(\mathbb{R}) \cap\overline{P(f)} \neq \emptyset$ and satisfying $(E_i \cdot \sum\limits_{j \neq i } E_j) \geq 3$. Then, by Theorem \ref{thm 8}, one knows that $e^{-\frac{2 i \pi \nu_i}{N_i}}$ is an eigenvalue of the monodromy $T^ * : H^*(\mathcal{F}_{f,0};\mathbb{C}) \to H^*(\mathcal{F}_{f,0};\mathbb{C})$. Moreover $E_i(\mathbb{R}) \cap \overline{P(f)} \neq \emptyset$ and $E_i(\mathbb{R})$ is contained in the real locus of $\sigma^{-1}(0)$, so $0 \in V_{+}(\mathbb{R})$. Now suppose that $s_0 = -\frac{1}{M_i}$ is induced by the strict transform $E_i$ of some $f_i$ such that $E_i(\mathbb{R}) \cap \overline{P(f)} \neq \emptyset$. Since $\sigma$ is an isomorphism outside the origin, it follows that the real locus of $f_i$
     is not reduced to the origin and that
     $$\{(x,y) \in \mathbb{R}^2 ~ |~ f_i(x,y)=0\}\cap \overline{\{(x,y) \in \mathbb{R}^2 ~ |~ f(x,y) > 0 \}} \neq \emptyset.$$
     One can therefore consider a point $a$ in the above set distinct from the origin. In particular, $a \in V_{+}(\mathbb{R})$ and $a \in \{ f_i = 0 \} \setminus \underset{j \neq i}{\cup} \{ f_j = 0 \}$, and by Proposition \ref{prop 8} one has $P_{a,0}(t)= t^{M_i}-1$ which completes the proof.
\end{proof}

\selectlanguage{english}

\bibliographystyle{amsalpha}
\bibliography{References.bib}

\bigskip


\end{document}